\documentclass[12pt]{article}
\usepackage{bm}
\usepackage[dvipsnames]{xcolor}
\newcommand{\colorblue}{}
\newcommand{\cov}{\operatorname{cov}}
\newcommand{\GL}{\operatorname{GL}}
\newcommand{\vv}[1]{{\mathbf#1}}
\usepackage[colo rlinks=true, allcolors=blue]{hyperref}

\usepackage{amsmath}
\usepackage{amsfonts}
\usepackage{amsthm}
\usepackage{graphicx}
\usepackage{amssymb, enumerate}
\usepackage{caption}
\usepackage{subcaption}
\usepackage{tikz}
\usepackage{a4wide}

\newtheorem{thm}{Theorem} [section]
\newtheorem{lem}[thm]{Lemma}
\newtheorem{prop}[thm]{Proposition}
\newtheorem{corr}[thm]{Corollary}

\theoremstyle{definition}

\theoremstyle{remark}

\newcommand{\C}{\mathbb{C}}
\newcommand{\Q}{\mathbb{Q}}
\newcommand{\Z}{\mathbb{Z}}
\newcommand{\N}{\mathbb{N}}
\newcommand{\R}{\mathbb{R}}

\numberwithin{equation}{section}
\newcommand{\sep}{e}
\newcommand{\sepKC}{\sep(L,\mathcal{C}_n)}

\newcommand{\sepnp}{\sep_{{\rm irr}}(n,p)}

\newcommand{\Pirrn}{\mathcal{P}_{{\rm irr}}(n)}
\newcommand{\Pirrstarn}{\mathcal{P}_{{\rm irr}}^*(n)}
\newcommand{\Palln}{\mathcal{P}(n)}
\newcommand{\Pallstarn}{\mathcal{P}^*(n)}
\newcommand{\Predn}{\mathcal{P}_{{\rm red}}(n)}
\newcommand{\Predstarn}{\mathcal{P}_{{\rm red}}^*(n)}
\newcommand{\sepall}{\sep_{{\rm all}}}
\newcommand{\sepirr}{\sep_{{\rm irr}}}
\newcommand{\sepred}{\sep_{{\rm red}}}
\newcommand{\sepallstar}{\sep^*_{{\rm all}}}
\newcommand{\sepirrstar}{\sep^*_{{\rm irr}}}
\newcommand{\sepredstar}{\sep^*_{{\rm red}}}

\begin{document}
\title{$p$-adic root separation and the discriminant of integer polynomials}
\author{Victor Beresnevich \and Bethany Dixon}

\date{}\maketitle

\begin{abstract}
In this paper we investigate the following related problems:
(A) the separation of $p$-adic roots of integer polynomials of fixed degree and bounded height; and
(B) counting integer polynomials of a fixed degree and bounded height with discriminant divisible by a (large) power of a fixed prime.
One of the consequences of our findings is the existence, for all large $Q>1$, of $Q^{2/n}$ integer irreducible polynomials $P$ of degree $n$ and height $\asymp Q$ with an almost prime power discriminant of maximal size, that is $|D(P)|\asymp Q^{2n-2}$ and $D(P)=p^kC_P$ with $C_P\in\Z$ satisfying $|C_P|\ll1$.
Our method generalises techniques developed for the real case and relies on a quantitative non-divergence estimate developed by Kleinbock and Tomanov.
\end{abstract}



\section{Introduction}

Throughout this paper $p \in \Z$ is a prime number and $n\in\N$.
A non-zero integer polynomial $P\in\Z[X]$ will be written as
$P = a_nx^n + \cdots + a_1x +a_0$ with $a_n\neq0$, where $n=\deg P$. If $P$ is monic, $a_n = 1$. Recall that the discriminant of $P$ is defined as
\begin{equation}\label{v103}
D(P):=a_n^{2n-2}\prod_{1\le i<j\le n}(\alpha_i-\alpha_j)^2\,,
\end{equation}
where $\alpha_1,\dots,\alpha_n$ are the roots of $P$ taken with multiplicity. Throughout we will use the standard height of $P$ defined by
\begin{equation}\label{height}
 H(P) := \max \{|a_0|,\dots,|a_n|\}\,.
\end{equation}
In this paper we address the $p$-adic case of the following broad and intricate problems (see \cite[Problem~52]{Bug07}, \cite{BBG16} and \cite{alge}):

\bigskip

\noindent\textbf{Problem A:} {\em Determine how close, as a function of height, distinct roots of a (monic) integer (irreducible) polynomial of a fixed degree $n\ge2$ can be.}

\bigskip

\noindent\textbf{Problem B:} {\it Find upper and lower bounds for the number of (monic) integer (irreducible) polynomials of degree $n\ge2$, bounded height and restricted discriminant.}

\medskip

In view of \eqref{v103}
the discriminant of a polynomial encodes the separation of the roots of a polynomial. The two problems we address in this paper are thus interrelated. In fact, the approach we adopt will allow us to make progress in both at once.

Questions on root separation as well as those pertaining to the ($p$-adic or real) size of the discriminant of integer polynomials, have been investigated for many decades as they are `embedded' in a variety of problems in Diophantine approximation and Transcendental and Algebraic number theory. For instance they underpin Sprind\v zuk's celebrated proof of Mahler's conjecture \cite{Sprindzuk:book_Mahler}, and are instrumental in various results on the famous (as yet open) conjecture of Wirsing from the 1960s on approximations by algebraic numbers \cite[Problem~2]{Bug07}. Counting monic polynomials of bounded height and
degree $n$ with arithmetic restrictions imposed on their discriminant, specifically with squarefree discriminant \cite{MR4419629}, has also been instrumental in some resent work such as \cite{MR4493242} on the classical problem of counting the number fields of fixed degree and bounded discriminant \cite{Schmidt-discriminants}. We note that in the case of \cite{MR4493242,MR4419629} the height is defined as the weighted version of \eqref{height} given by $H^*(P) := \max_{1\le i\le n}|a_{n-i}|^{1/i}$ for a monic $P$.

The $p$-adic case of Problem~A deals with the separation of the roots lying in the algebraic closure of the field $\Q_p$ of $p$-adic numbers.
In turn, the $p$-adic case of Problem~B seeks counting integer polynomials of bounded height and degree $n$ whose discriminant is divisible by a (large) power of $p$. In other words, the discriminant has a relatively small $p$-adic value.

We will discuss the state of the art on these problems and our new results in sections~\ref{Sep} and \ref{discr_back}. Subsequent sections will be solely dedicated to developing the techniques and establishing the results.

\section{Root separation: past and new results}\label{Sep}

To facilitate our discussion of Problem~A, we now introduce the exponents of root separation.
Within this section, $|\cdot|$ will denote either the real or $p$-adic absolute value on $\Q$, $K$ the completion of $\Q$ with respect to this absolute value.
Thus, $K=\Q_p$ if $|\cdot|=|\cdot|_p$ is the $p$-adic absolute value, and $K=\R=\Q_\infty$ if $|\cdot|=|\cdot|_\infty$.
Given a field $L$, $\overline{L}$ will stand for its algebraic closure. Let $\mathcal{C}_n$ be an infinite subclass of polynomials in $\Z[X]$ with $\deg P=n$. Suppose that $L$ satisfies $K\subset L\subset \overline{K}$. Define the root separation exponent $\sepKC$ as the infimum of all $\sep>0$ such that for all polynomials $P\in\mathcal{C}_n$ of sufficiently large height, the inequality
\begin{equation*}
 |\alpha_1-\alpha_2| > H(P)^{-\sep}
\end{equation*}
holds for any pair of distinct roots of $P$, $\alpha_1\neq\alpha_2$, lying in $L$.
In this paper we obtain lower bounds for
$$\sepnp:=\sep(\overline{\Q_p},\Pirrn)\,,
$$
where $\Pirrn$ is the set of all irreducible integer polynomials of degree $n$.
We note that $\sepnp$ is the largest real number such that for any $\sep<\sepnp$ we can find
 infinitely many $P\in\Pirrn$ such that
\begin{equation*}
 |\alpha_1-\alpha_2|_p \le H(P)^{-\sep}
\end{equation*}
holds for some roots $\alpha_1 \neq \alpha_2\in \overline{\Q_p}$ of $P$. Observe that the fact that $P$ is irreducible (over $\Q$) means that $\alpha_1$ and $\alpha_2$ are conjugate over $\Q$.

The root separation of integer polynomials has been intensively studied in the Archimedean case, in which the most understood exponents are
$$
\sep_{\mathrm{irr}}(n):=\sep(\C,\mathcal{P}_{\mathrm{irr}}(n))\qquad\text{and}\qquad \sep^*_{\mathrm{irr}}(n):=\sep(\C,\mathcal{P}_{\mathrm{irr}}^*(n))\,,
$$
where $\Pirrstarn$ is the set of all monic integer irreducible polynomials of degree $n$, as well as their analogues for all and all reducible integer polynomials:
$$
\sepall(n):=\sep(\C,\Palln)\,,\qquad \sepall^*(n):=\sep(\C,\Pallstarn)\,
$$
$$
\sep_{\mathrm{red}}(n):=\sep(\C,\Predn)\,,\qquad \sep_{\mathrm{red}}^*(n):=\sep(\C,\Predstarn)\,.
$$
Here
$\Palln$, $\Pallstarn$, $\Predn$, $\Predstarn$ are the sets of all, all monic, all reducible and all monic reducible integer polynomials of degree $n$ respectively. Below we provide a brief summary of known bounds:
\begin{itemize}
 \item Mahler \cite{Mah64} proved that $\sepirr(n) \le n-1$.
 \item Evertse \cite{eve04} proved that $\sepall(3) = 2$. An alternative proof of this was given in \cite{sch06}.
 \item Beresnevich, Bernik and G\"otze \cite{alge} found that\newline $\min\{\sepirr(n),\sepirrstar(n+1)\} \ge (n+1)/3$. Furthermore, it was proved in \cite{alge} that $\min\{\sep(\R,\mathcal{P}_{\mathrm{irr}}(n)),\sep(\R,\mathcal{P}^*_{\mathrm{irr}}(n+1))\}\ge (n+1)/3$.
 \item Bugeaud and Mignotte \cite{bug&mig11} proved the following results regarding general and irreducible polynomials:
 \begin{itemize}
 \item $\sepirr(2) = \sepall(2)=1$;
 \item $\sepirrstar(2) = \sepallstar(2)=0$;
 \item for any even integer $n \ge 4$, $\sepall(n) \ge \sepirr(n) \ge \frac n2$;
 \item for any odd integer $n \ge 5$, $\sepall(n) \ge \frac{n + 1}{2}$ and $\sepirr(n) \ge \frac{n + 2}{4}$;
 \item $\sepirrstar(3) = \sepallstar(3) \ge 3/2$ with equality if Hall's conjecture is true;
 \item for any even integer $n \ge 4$, $\sepallstar(n) \ge n/2$ and $\sepirrstar(n) \ge \frac{n-1}{2}$;
 \item for any odd integer $n \ge 5$, $\sepallstar(n) \ge \frac{n-1}{2}$ and $\sepirrstar(n) \ge \frac{n+2}{4}$.
 \end{itemize}
 \item Bugeaud and Dujella \cite{Bug&Duj11} obtained the following:
 \begin{itemize}
 \item for any integer $n \ge 4$, $\sepirr(n) \ge n/2 + \frac{n-2}{4(n-1)}$;
 \item for any odd integer $n \ge 7$, $\sepirrstar(n) \ge n/2 + \frac{n-2}{4(n-1)} -1$.
 \end{itemize}
 \item In a subsequent paper, Bugeaud and Dujella \cite{Bug&Duj14} proved that:
 \begin{itemize}
 \item for any even positive integer $n \ge 6$, $\sepallstar(n) \ge \frac{2n-3}{3}$;
 \item for any odd positive integer $n \ge 7$, $\sepallstar(n) \ge \frac{2n-5}{3}$;
 \item for any positive integer $n \ge 4$, $\sepirrstar(n) \ge \frac{n}{2} - \frac{1}{4}$.
 \end{itemize}
 \item Later, Dujella and Pejkovi\'{c} \cite{DP17} found new bounds for reducible monic polynomials of specific degrees:
 \begin{itemize}
 \item $\sepredstar(5) \ge \frac{7}{3}$;
 \item $\sepredstar(7) \ge \frac{17}{5}$;
 \item $\sepredstar(9) \ge \frac{31}{7}$.
 \end{itemize}
 \item For arbitrary degree, Dubickas \cite{Dub20} found that $\frac{n}{2} \le \sepred(n) \le \frac{3n-2}{4}$.
\end{itemize}

However, the results in the $p$-adic case are in short supply. Indeed, we are not aware of any relevant results except for the following two papers:
\begin{itemize}
 \item[$\star$]
For $n=3$, Pejkovi\'{c} \cite{pej16} found that if $p \neq 2$, $\sepirr(3,p) \ge 25/14$;
\item[$\star$] Bugeaud \cite{Bug16} investigated a related question regarding the distance between two different algebraic numbers, with one of them having a close conjugate.
\end{itemize}

Except \cite{alge}, the rest of the findings in the real case listed above rely on finding explicit polynomials with close roots.
Whether these constructions can be generalised to the $p$-adic case remains to be seen. In this paper, we build on the approach of \cite{alge}, which also enables quantitative bounds for the number of polynomials with close roots and produces counting results for Problem~B. But first we state our main non-quantitative result.

\begin{thm}\label{thm:short}
For any $n\ge 2$ and any prime $p$, we have that
$$
\sepirr(n,p) \ge \frac{n+1}{3}\,.
$$
\end{thm}

\medskip

\subsection{The quantitative theory}

Our quantitative results on Problem~A that will be stated below generalise those of \cite{alge} from the real case to the $p$-adics. Given $Q\ge1$, let
\begin{equation}\label{eqn1.2}
 \mathcal{P}_n(Q) := \{
 P \in \Z[X] : \deg(P) = n\text{ and }H(P) \le Q
 \}.
\end{equation}
Let $\theta \ge 0$, $Q\ge1$ and $C_0,C_1,C_2>0$. Define the following set
\newcommand{\set}{\mathbb{A}_{n}(Q, \theta, C_0,C_1,C_2)}
$$
 \set :=\hspace*{45ex}
$$
$$
=\left\{
 \alpha \in \Z_p :
 \begin{aligned}
 &\exists \text{ irreducible $P \in \Z[X]$ with $\deg P=n$,} \\
 &\text{$P(\alpha)=0$ and $C_1Q\le H(P)\le C_2Q$}\\
 &\text{such that $\exists\;\beta\in\overline{\Q_p}$ with $P(\beta)=0$}\\ &\text{and }
 0<|\alpha-\beta|_p \le C_0Q^{-\theta}
 \end{aligned}
 \right\}.
$$

In what follows $\mu$ will denote Haar measure on $\Q_p$ normalised so that $\mu(\Z_p)=1$.

\begin{thm}\label{Thm:Main2}
Let $n \ge 2$, $p$ be a prime, $0<\kappa<1$. Then there are constants $C_0,C_1,C_2>0$ depending on $n$, $p$, and $\kappa$ only such that the following property holds true.
For any $\theta$ satisfying
 \begin{equation}\label{eq:Beta_1_bound}
 0 \le \theta \le \frac{n+1}{3},
 \end{equation}
 and any ball
 $B = B(x_0,r) := \left\{
 x \in \Z_p : |x-x_0|_p \le r
 \right\}\subset\Z_p$ we have that
\begin{equation}\label{eq:measure_of_roots_intersecting_B_is_greater_than_3/4_of_B}
 \mu\left(\bigcup_{\alpha \in \set} B(\alpha,C_0Q^{-n-1+2\theta}) \cap B\right) \ge \kappa\mu(B)
 \end{equation}
for all sufficiently large $Q$.
\end{thm}

\medskip

\begin{corr}\label{cor:1}
Let $n \ge 2$, $p$ be a prime, $0<\kappa<1$. Then there are constants $C_0,C_1,C_2>0$ depending on $n$, $p$, and $\kappa$ only such that for any $\theta$ satisfying \eqref{eq:Beta_1_bound} and any ball $B\subset\Z_p$ \begin{equation}\label{eq:corr1}
 \# (\set \cap B) \ge \frac{\kappa}{p C_0} \cdot Q^{n+1-2\theta}\mu(B)
 \end{equation}
 for all sufficiently large $Q$.
\end{corr}
\begin{proof}
By a standard covering argument using the subadditivity of $\mu$, we have that
\begin{equation*}
\begin{aligned}
 &\# (\set \cap B) \cdot p C_0 Q^{-n-1+2\theta}\\
 &\ge
 \mu \left(
 \bigcup_{\alpha \in \set}
 B(\alpha, C_0Q^{-n-1+2\theta}) \cap B
 \right)
 \\
 &\ge \kappa\mu(B)
 \end{aligned}
 \end{equation*}
 where the final line comes about by \eqref{eq:measure_of_roots_intersecting_B_is_greater_than_3/4_of_B}. Now \eqref{eq:corr1} follows immediately.
\end{proof}

\begin{corr}\label{cor2}
Let $n \ge 2$. Then for all sufficiently large $Q$ there are $\gg Q^{\frac{n+1}{3}}$ $p$-adic algebraic numbers $\alpha\in\Z_p$ of degree $n$ and height $H(\alpha) \asymp Q$ such that
 \begin{equation}
 0<|\alpha-\beta|_p \;\ll\; Q^{-\frac{n+1}{3}}\quad\text{for some $\beta\in \overline{\Q_p}$ conjugate to $\alpha$\,,}
 \end{equation}
where the implied constants depend on $n$ and $p$ only.
\end{corr}
\begin{proof}
 This follows from Corollary \ref{cor:1} by taking $\theta = (n+1)/3$, $\kappa=1/2$ and $B=\Z_p$.
\end{proof}

Here and elsewhere $A\ll B$ means that $A\le CB$ for some $C>0$, which is referred to as the implied constant. We will also use the notation $A\asymp B$ which means $A\ll B\ll A$.

\medskip

\begin{proof}[Proof of Theorem \ref{thm:short}]
This follows immediately from Corollary \ref{cor2}.
\end{proof}

\section{Counting discriminants: past and new results}\label{discr_back}

As before, $n\ge2$, $Q>1$ and $\mathcal{P}_n(Q)$ is given by \eqref{eqn1.2}. It is well known that, for a polynomial $P$ of degree $n$, $D(P)$ is an integer polynomial of degree $2n-2$ in the coefficients of $P$, e.g. see \cite{BBG16}. This means that for every $P\in\Z[X]$ with $\deg P=n$, $D(P)\in\Z$ and
\begin{equation}\label{vb2.2}
|D(P)| \ll H(P)^{2n -2}\,,
\end{equation}
where the implied constant depends on $n$ only. Also, if $P$ does not have repeated roots then $|D(P)|\ge1$.
In particular, for any $P\in\Z[X]$ with $\deg P=n$ without repeated roots \begin{equation}\label{vb2.3}
H(P)^{-2(n -1)}\ll |D(P)|_p \le1\,.
\end{equation}
Therefore, in the context of Problem~B, one considers the following sets for $\nu\in [0, n-1]$\,:
\begin{align*}
\mathcal{D}_{n,\infty}(Q,\nu)&:=\big\{P \in \mathcal{P}_n(Q) : 1\le |D(P)| \ll Q^{2n-2-2\nu}\big\}\,,\\[0.5ex]
\mathcal{D}_{n,p}(Q,\nu)&:=\big\{P \in \mathcal{P}_n(Q) : 0<|D(P)|_p \ll Q^{-2\nu}\big\}\,,
\end{align*}
where the implied constants depend on $n$ and $p$ only.
For $v=\infty$ and $v=p$
we also define
\begin{align*}
\mathcal{D}^{\rm irr}_{n,v}(Q,\nu)& := \mathcal{D}_{n,v}(Q,\nu) \cap \big\{\text{$P$ is irreducible over } \Q\big\}\,.
\end{align*}

\subsection{Previous results}

We begin with a survey of known results for $v=\infty$. The first explicit bound on $\#\mathcal{D}_{n,\infty}(Q,\nu)$ was established by Bernik, G\"otze and Kukso \cite[Theorem~1]{BGO08}, who showed that
\begin{equation}\label{vb3.3}
\#\mathcal{D}_{n,\infty}(Q,\nu) \gg Q^{n+1-2\nu}\qquad\text{for }\nu \in [0,\tfrac{1}{2}]\,.
\end{equation}
This was later extended in \cite{MR2808979} to $\nu\in[0,(n-2)/3]$.

Using counting results on rational points near curves \cite{MR2373145,MR2242634}
it was shown in \cite{BBG16} that
\begin{equation*}
\#\mathcal{D}_{2,\infty}(Q,\nu) \asymp Q^{3-2\nu}\qquad\text{for all }\nu \in [0,1)\,.
\end{equation*}
In particular, this means that \eqref{vb3.3} is sharp for $n=2$. Furthermore, an asymptotic formula for $\#\mathcal{D}_{2,\infty}(Q,\nu)$ was obtained in \cite{GKK13} for $0\le\nu<\tfrac34$. However, for $n\ge3$, \eqref{vb3.3} turned out to be far from the truth. Indeed, G\"otze, Kaliada and Kukso \cite{GKK14} proved that $$
\#\mathcal{D}_{3,\infty}(Q,\nu) \asymp Q^{4- \frac{5}{3}\nu}
$$
for $0 \le \nu < \frac{3}{5}$, and they also established an asymptotic formula.
For any $n\ge2$, Beresnevich, Bernik and G\"otze \cite{BBG16} obtained the following lower bound for all $0\le \nu\le n-1$:
\begin{equation}\label{BBG16bound}
\#\mathcal{D}_{n,\infty}(Q,\nu) \gg Q^{n+1-\frac{n+2}{n}\nu}\,.
\end{equation}
This is believed to be optimal. Recently, Badziahin \cite[Theorem~4]{Badz} completed the story for the cubic case ($n=3$) by showing that for any $\nu\in[0,2]$ and $\varepsilon>0$
\begin{equation}\label{Badz}
\#\mathcal{D}_{3,\infty}(Q,\nu) \ll Q^{4- \frac{5}{3}\nu+\varepsilon}
\end{equation}
for sufficiently large $Q$.
No other generic upper bounds for $\#\mathcal{D}_{n,\infty}(Q,\nu)$ are known; however there are several results with additional constraints on the distribution of roots \cite{BBG17,BBO15, BB0'D2020, bud19}.

Now we turn to the $p$-adic case, in which little is known. Bernik, G\"otze and Kukso \cite{BGK08} proved that
$$
\#\mathcal{D}_{n,p}(Q,\nu) \gg Q^{n+1-2\nu}\qquad\text{for $0\le \nu \le \tfrac{1}{2}$}\,,
$$
which is analogous to \eqref{vb3.3}.
Very recently, generalising \eqref{Badz}, Bernik, Vasilyev, Kudin and Panteleeva \cite[Theorem~4]{BVKP24} gave the following upper bound for $n=3$:
\begin{equation}\label{BVKP24}
\#\mathcal{D}_{3,p}(Q,\nu) \ll Q^{4-\frac53\nu+\varepsilon}\qquad\text{for $0\le \nu \le 2$}\,.
\end{equation}

\subsection{New results}

In this paper we establish the following lower bound generalising
the main result of Beresnevich, Bernik and G\"otze \cite{smallDisciminants} to the $p$-adic case:

\begin{thm}\label{Thm:Main}
Let $n\ge 2$ be an integer, $p$ be a prime. Then, for any $0\le\nu\le n-1$
 \begin{equation}
 \# \left(\mathcal{D}^{\rm irr}_{n,p}(Q,\nu)\cap\big\{P\in\Z[X]: H(P)\asymp Q\big\}\right) \;\gg\; Q^{n+1 - \frac{n+2}{n}\nu}
 \end{equation}
for all sufficiently large $Q$, where all implied constants depend on $n$ and $p$ only.
\end{thm}

\smallskip

\begin{corr}[Almost prime power discriminants]
For any $n\ge2$ and sufficiently large $Q$ there are $\gg Q^{2/n}$ integer irreducible polynomials $P$ of degree $n$ and height $H(P)\asymp Q$ such that for some $k=k(P)\in\N$ and $C=C(P)\in\Z$ we have that
$$
|D(P)|\asymp Q^{2n-2}\,,\quad D(P)=p^kC\quad\text{and}\quad|C|\ll1\,,
$$
where the implied constants depend on $n$ and $p$ only.
\end{corr}

\medskip

Combining Theorem~\ref{Thm:Main}
with \eqref{BVKP24} we get the following

\medskip

\begin{corr}[The cubic case]
Let $n=3$, $p$ be any prime. Then, for any $0\le\nu\le 2$ and any $\varepsilon>0$, for all sufficiently large $Q$ we have that
 \begin{equation}
 1\;\ll\;\# \mathcal{D}_{3,p}(Q,\nu) \cdot Q^{-(4 - \frac{5}{3}\nu)}\;\ll\; Q^{\varepsilon}\,,
 \end{equation}
where all implied constants depend on $n$ and $p$ only.
\end{corr}

\subsection{Further remarks}

In this subsection, we present a general problem that extends the questions we have discussed above to the case of several primes. Let $S$ be a non-empty finite set that contains only prime numbers and $\infty$. Let $\boldsymbol{\nu}_S=(\nu_v)_{v\in S}$ be a vector of non-negative reals. Define
$$
\mathcal{D}_{n,S}(Q,\boldsymbol{\nu}_S):=\bigcap_{v\in S}\mathcal{D}_{n,v}(Q,\nu_v)\qquad\text{
and}\qquad
\mathcal{D}^{\rm irr}_{n,S}(Q,\boldsymbol{\nu}_S):=\bigcap_{v\in S}\mathcal{D}^{\rm irr}_{n,v}(Q,\nu_v)\,.
$$

\medskip

\noindent\textbf{Main Problem:} {\em With $\mathcal{D}^\circ_{n,S}(Q,\boldsymbol{\nu}_S)$ standing for either
$\mathcal{D}_{n,S}(Q,\boldsymbol{\nu}_S)$ or $\mathcal{D}^{\rm irr}_{n,S}(Q,\boldsymbol{\nu}_S)$,
verify for any $n\ge2$ and $S$ and $\boldsymbol{\nu}_S$ as above such that
$$
\nu:=\sum_{v\in S}\nu_v\le n-1
$$
for any $\varepsilon>0$ and all sufficiently large $Q$
$$
Q^{n+1-\frac{n+2}n\nu}\;\;\ll\;\; \#\mathcal{D}^\circ_{n,S}(Q,\boldsymbol{\nu}_S) \;\;\ll\;\; Q^{n+1-\frac{n+2}n\nu+\varepsilon}\,.
$$
}

\bigskip

Little is known about the general case for $\#S\ge2$. However, Bernik, Budarina and O'Donnell \cite{BBO18} established that when $n=3$ and $S=\{\infty,p\}$, for any $\varepsilon>0$ we have that
$$
\#\mathcal{D}_{3,S}(Q,\boldsymbol{\nu}_S) \ll Q^{4-\frac{5}{3}(\nu_\infty+\nu_p)+\varepsilon}
$$
holds for all sufficiently large $Q$ if $\frac{3\varepsilon}{20}\le \nu_\infty+\nu_p \le \frac{6}{5}$.
In turn, Budarina, Dickinson and Yuan \cite{BDY12} verified that if $n\ge3$, $S=\{\infty,p\}$ and $\boldsymbol{\nu}=(\nu,\nu)$, that is $\nu_\infty=\nu_p=\nu$, then
$$
\#\mathcal{D}_{n,S}(Q,\boldsymbol{\nu}) \gg Q^{n+1-4\nu}\qquad \text{for $0\le \nu \le \tfrac{1}{3}$.}
$$

\section{Key Lemma on Polynomials}\label{sec2}

{\colorblue{}We begin by stating the following lemma, which is instrumental in establishing the main results of this paper. It allows us to construct many irreducible polynomials whose height and derivatives have prescribed sizes.

\begin{lem}\label{lem:aux_lemma}
Let $n \ge 2$ be an integer, $p$ be a prime, $0<v<1$, $0<\kappa<1$, and let
\begin{equation}\label{theball}
 B := B(x_0,r) = \left\{
 x \in \Z_p : |x-x_0|_p \le r
 \right\},
\end{equation}
where $x_0 \in \Z_p$ and $0 < r \le 1$.
Then there exist positive constants $\delta_0$, $C_1$ and $C_2$ depending on $n$, $p$, and $\kappa$ only, and a constant $Q_0>0$ depending on $B$, $n$, $p$, $v$, and $\kappa$ only, such that the following holds.

For any $Q\ge Q_0$ and any parameters
\begin{equation}\label{xiorder}
0 < \xi_0 \le \dots \le \xi_{n-1} \le \xi_n = 1
\end{equation}
satisfying
\begin{equation}\label{xiorder+}
\prod_{i=0}^n \xi_i = Q^{-(n+1)}
\qquad\text{and}\qquad
\xi_0\le Q^{-1-v},
\end{equation}
there exists a measurable set $G_B \subset B$, depending on $n$, $p$, $B$, $\kappa$, $Q$, and $\xi_0,\dots,\xi_n$, such that
\begin{equation}\label{muB_G}
 \mu(G_B) \ge \kappa\mu(B).
\end{equation}
Moreover, for every $x \in G_B$ there are $n+1$ linearly independent primitive irreducible polynomials $P \in \Z[X]$ of degree $n$ satisfying
\begin{equation}\label{Height}
C_1Q\le H(P)\le C_2Q
\end{equation}
and
\begin{equation}\label{eq:aux_lemma_statement}
 \delta_0\xi_i \le
 \left|\frac{1}{i!}P^{(i)}(x)\right|_p
 \le \xi_i\quad(0 \le i \le n),
\end{equation}
where $P^{(i)}(x)$ denotes the $i$-th derivative of the polynomial $P$ at $x$.
\end{lem}

In the remainder of this section, we construct the set $G_B$, and for each $x\in G_B$ we construct primitive irreducible polynomials $P$ of degree $n$ satisfying \eqref{Height} and \eqref{eq:aux_lemma_statement}. The proof of \eqref{muB_G} relies on the so-called quantitative non-divergence estimate and will be given in Section~\ref{AuxLemSection}, after we recall the estimate and establish the necessary auxiliary results.}

\subsection{Auxiliary statements} \label{subsection:System of equations and the body}

{\colorblue{}Our first observation is that, in proving Lemma~\ref{lem:aux_lemma}, it suffices to assume that the parameters $\xi_0,\dots,\xi_n$ and $Q$ are integer powers of $p$. Indeed, suppose that we are given parameters $\xi_i$ and $Q>1$ satisfying \eqref{xiorder} and \eqref{xiorder+}.
Then, for $Q$ sufficiently large depending on $n$, we can choose integers $b_i \in \Z_{\ge0}$ such that
\begin{equation}\label{eqn2.5+}
p^{-b_i} \le \xi_i \le p^{-b_i+n}, \qquad
p^{-b_0} \le p^{-b_1} \le \dots \le p^{-b_n} = 1
\end{equation}
and
\begin{equation}\label{eqn2.5}
\sum_{i=0}^n b_i = t(n+1)
\end{equation}
for some $t \in \mathbb{N}$.
Define $\tilde{\xi}_i = p^{-b_i}$, $\tilde{Q} = p^t$, and $\tilde{v}=v/2$. Then
\[
p^{-n}\xi_i \le \tilde{\xi}_i \le \xi_i,
\]
and $\tilde{Q}/Q$ is bounded above and below by positive constants depending only on $n$ and $p$.

It is straightforward to see that the modified parameters $\tilde{Q}, \tilde{\xi}_0,\dots,\tilde{\xi}_n$, and $\tilde{v}$ satisfy \eqref{xiorder} and \eqref{xiorder+} for all sufficiently large $Q$.
Moreover, if the conclusion of Lemma~\ref{lem:aux_lemma} holds for the parameters $\tilde{Q},\tilde{\xi}_0,\dots,\tilde{\xi}_n,\tilde{v}$, then it also holds for the original parameters $Q,\xi_0,\dots,\xi_n,v$, up to adjusting the constants $Q_0,C_1,C_2,\delta_0$ with factors depending on $n$ and $p$ only.

Thus, while proving Lemma~\ref{lem:aux_lemma}, we may assume without loss of generality that
\begin{equation}\label{eqn2.6}
\xi_i = p^{-b_i} \qquad\text{and}\qquad Q = p^t\qquad\text{for some $b_i,t \in \Z_{\ge0}$.}
\end{equation}

\medskip

The following relatively well-known statement (cf. Lemma~2.2.2 in \cite{DattaGhosh2022}) will be required to use Minkowski's theorem for convex bodies in order to find integer polynomials $P$ of bounded height satisfying the right hand-side of \eqref{eq:aux_lemma_statement}.}

\begin{prop}\label{prop2.2}
Let $x\in\Z_p$ and $\xi_i$ be given by \eqref{eqn2.6}. Let $\Gamma$ be the set of integer points $(a_0,\dots,a_n)$ such that the polynomial $P=a_nX^n+\dots+a_0$ satisfies
\begin{equation} \label{polybounds oldNEW}
 \left|\frac{1}{i!} P^{(i)}(x)\right|_p \le \xi_i \qquad (0 \le i \le n)\,.
\end{equation}
Then $\Gamma$ is a sublattice of $\Z^{n+1}$ such that
\begin{equation}\label{eqn2.12}
 \operatorname{cov}(\Gamma) = \prod_{i=0}^n \xi_i^{-1}.
\end{equation}
\end{prop}

\begin{proof}
The proof is elementary, but we give a brief argument for completeness.
Since $\Z$ is dense in $\Z_p$, we may assume without loss of generality that $x$ in \eqref{polybounds oldNEW} belongs to $\Z$. Then, the quantities $(i!)^{-1}P^{(i)}(x)$ are also in $\Z$ for any integer polynomial $P$. Hence, by \eqref{eqn2.6}, the system \eqref{polybounds oldNEW} is equivalent to the system $\frac{1}{i!}P^{(i)}(x) \equiv 0 \pmod{p^{b_i}}$ $(0 \le i \le n)$, which in turn can be written as
\begin{equation}\label{explicit}
\begin{pmatrix}
 1 & x & x^2 & \cdots & x^n\\
 0 & 1 & 2x & \cdots & nx^{n-1} \\
 0 & 0 & 1 & \cdots & \frac{1}{2}n(n-1)x^{n-2} \\
 \vdots & \vdots & \vdots & \ddots & \vdots \\
 0 & 0 & 0 & \cdots & 1
\end{pmatrix}
\begin{pmatrix}
 a_0 \\ a_1 \\ a_2 \\ \vdots \\ a_n
\end{pmatrix}=\begin{pmatrix}
 k_0p^{b_0} \\ k_1p^{b_1} \\ k_2p^{b_2} \\ \vdots \\ k_np^{b_n}
\end{pmatrix}
\end{equation}
for some $k_i\in\Z$, where $a_0,\dots,a_n$ are regarded as the coefficients of $P$, as in the statement. The set of points on the right of \eqref{explicit}, taken over all $k_0,\dots,k_n\in\Z$, is easily seen to be a sublattice of $\Z^{n+1}$, say $\Gamma_0$, of covolume $\prod_{i=0}^n p^{b_i}=\prod_{i=0}^n \xi_i^{-1}$. The matrix on the left of \eqref{explicit}, say $T$, has integer entries and a determinant of $1$. Hence $T$ has an inverse over $\Z$, and multiplying \eqref{explicit} on both sides by $T^{-1}$ gives an explicit parametrisation of $\Gamma$, which is $\Gamma=T^{-1}\Gamma_0$.
In particular, it means that $\Gamma$ is a sublattice of $\Z^{n+1}$ and
\begin{equation*}
 \operatorname{cov}(\Gamma) = \det T^{-1}\operatorname{cov}(\Gamma_0) = \prod_{i=0}^n \xi_i^{-1}
\end{equation*}
as stated.
\end{proof}

{\colorblue{}

\subsection{The set $G_B$ and the proof of \eqref{Height} and \eqref{eq:aux_lemma_statement}} \label{Main Lemma}

Our approach is based on \cite{alge}. Let $n,p,\kappa,v,\xi_0,\dots,\xi_n$ and $Q$ be as in Lemma~\ref{lem:aux_lemma}; in particular, \eqref{xiorder} and \eqref{xiorder+} are satisfied. Furthermore, as explained above, we assume without loss of generality that $\xi_i$ and $Q$ are integer powers of $p$, that is, \eqref{eqn2.6} holds.

Let 
$$
B_{Q}:=\{\vv a=(a_0,\dots,a_n)\in\R^{n+1}: \max_{0\le i\le n}|a_i|\le Q\}.
$$
Clearly, $B_Q$ is a convex body in $\R^{n+1}$ symmetric about the origin.} Furthermore,
\begin{equation}\label{eqn2.11}
\operatorname{vol}(B_{Q})=(2Q)^{n+1}\,.
\end{equation}
Let $\Gamma$ be the lattice as in Proposition~\ref{prop2.2}, and let $\lambda_1,\dots,\lambda_{n+1}$ be the successive minima of $B_{Q}$ with respect to $\Gamma$, that is
$$
\lambda_i:=\inf\big\{\lambda>0:\operatorname{rank}\big(\Gamma\cap(\lambda B_{Q})\big)\ge i\big\}\,.
$$
{\colorblue{}Note that the lattice $\Gamma$ and consequently each $\lambda_i$ depend on the choice of $x\in\Z_p$.}
By \eqref{eqn2.12}, \eqref{eqn2.11} and Minkowski's second theorem for convex bodies, we get that
\begin{equation}
 (2Q)^{n+1}
 \prod_{i=1}^{n+1}\lambda_i
 \le
 2^{n+1}\left({\prod_{i=0}^n \xi_i}\right)^{-1}.
\end{equation}
Hence, by \eqref{xiorder+} and the inequalities $\lambda_1\le\dots\le\lambda_{n+1}$, we get that
\begin{equation}\label{eq:minima_bound}
 \lambda_1^{n} \lambda_{n+1}
 \le
 \prod_{i=1}^{n+1}\lambda_i
 \le
 Q^{-(n+1)} \left({\prod_{i=0}^n \xi_i}\right)^{-1} = 1\,.
\end{equation}
Define
\begin{equation}\label{vb6.4}
 E(B;\varepsilon_0) = \{x\in B : \lambda_1 \le \varepsilon_0\}\,,
\end{equation}
where {\colorblue{}$B$ is as in \eqref{theball} and} $\varepsilon_0>0$ is a small parameter, to be determined later. Suppose that $x \in B\setminus E(B; \varepsilon_0)$. Then $\lambda_1 > \varepsilon_0$. Combining this with \eqref{eq:minima_bound} gives
\begin{equation}
 \lambda_{n+1}
 \le
 c_0:=(\varepsilon_0)^{-n}.
\end{equation}
Hence, by the definition of $\lambda_{n+1}$, there are $n+1$ linearly independent polynomials $P_j=a_{n,j}X^n+\dots+a_{0,j}\in\Z[X]$ for $0 \le j \le n$ satisfying \eqref{polybounds oldNEW} and
\begin{equation}\label{eqn4.8}
 \max_{0\le i\le n}|a_{i,j}|\le c_0Q
\end{equation}
{\colorblue{} for each $j$. We now modify the polynomials $P_j$ so as to obtain irreducible
polynomials while preserving the required $p$-adic estimates.}
Define the sublattice $\Lambda$ of $\Gamma$ as the $\Z$-span of $\vv a_j=(a_{0,j},\dots,a_{n,j})^T$ for $0\le j\le n$, where ${}^T$ means transposition. Then
\begin{equation*}
 \operatorname{cov}(\Lambda) = m \cdot \operatorname{cov}(\Gamma)\,,
\end{equation*}
where $m \in \N$ is the index of $\Lambda$ in $\Gamma$. {\colorblue{}Since the basis of $\Lambda$} can be chosen to be contained in the body defined by \eqref{eqn4.8}, {\colorblue{}each basis vector has Euclidean norm $\le\sqrt{n+1}c_0Q$}, and we have that
\begin{equation*}
 \operatorname{cov}(\Lambda) \le ({\colorblue{}\sqrt{n+1}}c_0Q)^{n+1} = ({\colorblue{}\sqrt{n+1}}c_0)^{n+1}\operatorname{cov}(\Gamma)\,,
\end{equation*}
where the latter follows from \eqref{xiorder+} and \eqref{eqn2.12}.
Hence, $m \le ({\colorblue{}\sqrt{n+1}}c_0)^{n+1}$. Choose a prime number {\colorblue{}$q$ such that $m < q \le 8m$} with $q \neq 2$ or $p$.
{\colorblue{}The interval $(m,8m]$ is sufficiently wide to contain at least three primes by Bertrand's Postulate, ensuring that a prime $q$ different from $2$ and $p$ can be chosen.}

Let $A=(a_{i,j})_{0\le i,j\le n}$ be the matrix whose columns are the vectors $\vv a_j$ $(0\le j\le n)$. Then $1\le |\det A|=\cov(\Lambda)=m\cov(\Gamma)$ and since $\cov(\Gamma)=Q^{n+1}$ is an integer power of $p$ and $q>m$, $q$ does not divide $\cov(\Lambda)$. Therefore $q$ does not divide $\det A$, and {\colorblue{}$A$ has an inverse modulo $q$. Thus, we can define an integer matrix $A^{-1}$ whose entries are integers in $[0,q-1]$ such that $A^{-1}A\equiv AA^{-1}\equiv I_{n+1}\pmod{q}$, where $I_{n+1}$ is the identity matrix. Hence,
\begin{equation}\label{AA}
AA^{-1}=I_{n+1}+qM
\end{equation}
for some integer matrix $M=(m_{i,j})_{0\le i,j\le n}$.
Let $S=(s_{i,j})_{0\le i,j\le n}$ be an $(n+1)\times(n+1)$ integer matrix with $1$'s in the last row and zeros elsewhere. That is,
\begin{equation}\label{S_conditions}
s_{n,j}=1,\;\; s_{i,j}=0\quad
(0\le i\le n-1,0\le j\le n).
\end{equation}
Let $R$ be an $(n+1)\times(n+1)$ integer matrix and define
\begin{equation}\label{def_H}
H=A^{-1}(S+qR).
\end{equation}
Then, using \eqref{AA}, we get that 
\begin{equation}\label{H}
AH\equiv S\pmod{q}.
\end{equation}
Furthermore, using \eqref{AA} and \eqref{def_H}, we get that
\begin{align}
\nonumber AH-S&=AA^{-1}(S+qR)-S
=(I_{n+1}+qM)(S+qR)-S\\
&=qMS+qR+q^2MR.\label{new_equation3}
\end{align}
Let
\begin{equation}\label{tildeH}
\tilde A=(\tilde a_{i,j})_{0\le i,j\le n}:=AH.
\end{equation}
By \eqref{S_conditions}, the first row of $\tilde A$ is the same as the first row of $AH-S$. Hence, by \eqref{S_conditions} and \eqref{new_equation3}, for each $j=0,\dots,n$ we have that
\begin{equation}\label{AH-S}
\tilde a_{0,j}\equiv q(m_{0,n}+r_{0,j})\pmod{q^2}.
\end{equation}
For each $j$ choose $r_{0,j}\in\{-1,0,1\}$ such that 
\begin{equation}\label{condition1}
m_{0,n}+r_{0,j}\not\equiv 0\pmod{q}\qquad(0\le j\le n)
\end{equation}
and
\begin{equation}\label{condition1+}
(r_{0,0},\dots,r_{0,n})\neq\pm(1,\dots,1)\quad\text{and}\quad(r_{0,0},\dots,r_{0,n})\neq(0,\dots,0).
\end{equation}
Next, let $r_{n,j}=0$ for $0\le j\le n$, and for $1\le i\le n-1$ choose $r_{i,j}\in\{0,1\}$ so that
\begin{equation}\label{condition2}
\det(S+qR)\neq0.
\end{equation}
The existence of such $r_{i,j}$ follows immediately from 
\eqref{S_conditions} and \eqref{condition1+}.

By \eqref{def_H}, \eqref{condition2}, and the fact that $\det A\neq0$ and $\det A^{-1}\neq0$, we have that $\det \tilde A\neq0$, where $\tilde A$ is defined by \eqref{tildeH}. Consequently, the integer polynomials $\tilde P_j:=\tilde{a}_{n,j}X^n+\dots+\tilde{a}_{0,j}$ are linearly independent.

By \eqref{S_conditions} and \eqref{H}, for each $j$ we have that $\tilde{a}_{i,j} \equiv 0 \pmod{q}$ for $0 \le i \le n-1$, $\tilde{a}_{n,j} \not\equiv 0 \pmod{q}$. Furthermore, by \eqref{AH-S} and \eqref{condition1}, we have that $\tilde{a}_{0,j} \not\equiv 0\pmod{q^2}$ $(0\le j\le n)$.
Therefore, $\deg \tilde P_j=n$ and, by Eisenstein's criterion, $\tilde P_j$ is irreducible for all $0\le j \le n$.

The height of $\tilde P_j$ can be estimated using \eqref{tildeH} by first obtaining an upper bound on the entries $h_{i,j}$ of $H$. 
By the choice of $A^{-1}$, $S$ and $R$, the entries of $A^{-1}$ and of $S+qR$ are bounded by $q$ in absolute value. Hence, by \eqref{def_H} and the fact that $q\le 8m$,
\begin{equation}\label{eta-bound}
 |h_{i,j}| \le (n+1)q^2 \le (n+1)(8m)^2.
\end{equation}
Choose the smallest $C_2\ge c_0(n+1)^2(8m)^2$ satisfying \begin{equation}\label{c2}
C_2=p^{2u}\qquad\text{for some $u\in\Z_{\ge0}$.}
\end{equation}
Then, by \eqref{eqn4.8} and \eqref{eta-bound}, we get from \eqref{tildeH} that
\begin{equation}\label{Upper2}
 \max_{0\le i \le n}|\tilde{a}_{i,j}| \le C_2Q
\end{equation}
for each $j$. Also, by construction, the coefficients of every polynomial $\tilde P_j$ are in $\Lambda\subset\Gamma$ and, hence, the upper bounds in \eqref{eq:aux_lemma_statement} hold for the polynomials $\tilde P_j$. Furthermore, recall that $E(B; \varepsilon_0)$ is given by \eqref{vb6.4}. Then,
by the assumption 
$x \in B\setminus E(B; \varepsilon_0)$, we have the lower bound in \eqref{Height} holds with $C_1=\varepsilon_0$ for each polynomial $\tilde P_j$. Together with \eqref{Upper2}, this verifies \eqref{Height} in full for each of the polynomials $\tilde P_j$, as well as the upper bounds in \eqref{eq:aux_lemma_statement}.

Now we turn to establishing the lower bounds in \eqref{eq:aux_lemma_statement}. This is done by imposing further restrictions on $x$.
For each $j=0,\dots,n$, consider the inequalities
\begin{equation}\label{eq:improved_Polys}
 \left|\frac{1}{i!} P^{(i)}(x)\right|_p \le \delta_i^j\xi_i\,,
\end{equation}
where
\begin{equation}\label{eq:delta_def}
 \delta_i^j = \begin{cases}
 \delta_0 \qquad &\text{if } i=j, \\
 1 \qquad &\text{otherwise}\,.
 \end{cases}
\end{equation}}
For each $j$ define
\begin{equation}\label{Ej}
 E_{j}(B;\delta_0):=
 \left\{
 x \in B : \begin{aligned}
 &\text{$\exists$ } P \in\Z[X]\text{ with }
 \deg(P)=n \\ &
 \text{ and } H(P)\le C_2Q \text{ such that}\\ & \text{inequalities \eqref{eq:improved_Polys} hold}
 \end{aligned}\right\}\,.
\end{equation}
Now let
\begin{equation}\label{G_B}
 G_B :=B \setminus \left(
 \bigcup^n_{j=0}E_j(B;\delta_0)
 \cup E(B;\varepsilon_0)
 \right)\,.
\end{equation}
{\colorblue{}Then, for any $x\in G_B$, the polynomials $\tilde P_j$ we have constructed necessarily satisfy the lower bounds in \eqref{eq:aux_lemma_statement}. 
This is because the removal of $x$ from $E_j(B;\delta_0)$ for each $j$ forces the inequality $|\frac1{j!}P^{(j)}(x)|_p>\delta_0\xi_j$ for every $j$ and every non-zero integer polynomial $P$ of height $\le C_2Q$.

In summary, for every $x\in G_B$, we have constructed $n+1$ linearly
independent irreducible polynomials $\tilde P_j\in\Z[X]$ of degree $n$ satisfying
\eqref{Height} and \eqref{eq:aux_lemma_statement}. Moreover, it follows
from the construction that the constants $C_1$ and $C_2$ depend only on
$n$, $p$, and $\varepsilon_0$.

Now we demonstrate that the polynomials $\tilde P_j$ can be taken to be primitive, that is, their coefficients will have no common divisor $>1$. First of all, we may assume the coefficients of $\tilde P_j$ have no common prime divisor different from $p$, since otherwise we can divide through by that divisor without affecting \eqref{eq:aux_lemma_statement} or the upper bound on the height. The lower bound on height will also be preserved for the same reason as we have given above, that is that $x$ is removed from $E(B;\varepsilon_0)$. Now suppose that the coefficients of $\tilde P_j$ are divisible by $p^{l_j}$ for some $l_j\in\N$ depending on $\tilde P_j$. Then, by \eqref{eq:aux_lemma_statement} for $i=n$ and the assumption $\xi_n=1$, see \eqref{xiorder}, we get that $p^{l_j}\le \delta_0^{-1}$. Then, cancelling $p^{l_j}$ from the coefficients of $\tilde P_j$ for each $j$ gives a collection of $n+1$ primitive irreducible integer polynomials $P$ satisfying
\begin{equation}\label{eq:aux_lemma_statement2}
 \delta_0\xi_i \le
 \left|\frac{1}{i!} P^{(i)}(x)\right|_p
 \le \delta_0^{-1}\xi_i\quad(0 \le i \le n)
\end{equation}
and $\delta_0C_1Q\le H(P)\le C_2 Q$. To return to system
\eqref{eq:aux_lemma_statement}, in which we have $\xi_i$ without any constant factor on the right hand-side, we simply apply the above construction to the following modified initial parameters: 
$$
\xi_i^*=\delta_0\xi_i\;(0\le i\le n-1),\quad \xi_n^*=1,\quad Q^*=\delta_0^{-n/(n+1)}Q,\quad v^*=v/2,
$$ 
which also satisfy \eqref{xiorder} and \eqref{xiorder+} for all sufficiently large $Q$ making the above construction applicable. As a result, we will obtain $n+1$ primitive irreducible linearly independent polynomials of degree $n$ satisfying exactly \eqref{eq:aux_lemma_statement} and $C_1 Q\le H(P)\le C_2Q$, up to adjusting the constants $Q_0,C_1,C_2,\delta_0$ with factors depending on $n$, $p$, and $\delta_0$ only.

\medskip

To complete the proof of Lemma~\ref{lem:aux_lemma} it remains to
\begin{itemize}
 \item[(i)] verify the measure estimate \eqref{muB_G} for $Q\ge Q_0$ for a suitably chosen $Q_0$,
 \item[(ii)] show that $\varepsilon_0$ and $\delta_0$ can be chosen to depend on $n$, $p$, and $\kappa$ only.
\end{itemize}
Both of these tasks require a quantitative non-divergence estimate, which we now introduce.
}

\section{A quantitative non-divergence estimate} \label{AuxLemSection}

\subsection{A result of Kleinbock and Tomanov}\label{KT0}

{\colorblue{}We will use the following statement from \cite{flows}, whose notation will be explained immediately after the theorem.}

\begin{thm}[Theorem 9.3 of \cite{flows}]\label{thm:Original_aux}
Let $X$ be a Besicovitch metric space, $\mu$ a uniformly Federer measure on $X$, and let $S$ be a finite collection of valuations of $\Q$ including the Archimedean one.
Let $m \in \N$, and let a ball $B=B(x_0,r_0) \subset X$ and a continuous map 
\begin{equation}\label{GLQS}
{\colorblue{}
h:\tilde{B} \rightarrow \GL(m,\Q_S):=\prod_{v\in S}\GL(m, \Q_v)}
\end{equation}
be given,
where $\tilde{B}=B(x_0,3^{m}r_0)$.
Now suppose that for some $C,\alpha >0$ and $0<\rho <1$ one has
\begin{enumerate}
\item[{\rm(1)}] for all $\Delta\in \mathfrak{B}(\Z_S,m)$, the function $\operatorname{cov}(h(\cdot)\Delta)$ is $(C,\alpha)$-good on $\tilde{B}$ with respect to $\mu$;
\item[{\rm(2)}] for all $\Delta\in \mathfrak{B}(\Z_S,m)$, $\|\operatorname{cov}(h(\cdot )\Delta)\|_{\mu,B} \ge \rho$.
\end{enumerate}
Then, for every $0<\varepsilon \le \rho$,
\begin{equation}
\mu \left( \{x \in B : \delta(h(x)\Z_S^{m})\le\varepsilon\}\right)
\le
mC\left( N_X D_\mu^2 \right)^{m} \left(\frac{\varepsilon}{\rho}\right)^\alpha \mu(B).
\end{equation}
\end{thm}

{\colorblue{}We note that the original statement in \cite[Theorem 9.3]{flows} is given with the strict inequality $\delta(h(x)\Z_S^{m})<\varepsilon$. However, replacing the strict inequality with a non-strict inequality is simple. One can take $\varepsilon'>\varepsilon$, apply the original statement with $\varepsilon'$, and then let $\varepsilon'\to\varepsilon$.}

The definitions used in Theorem~\ref{thm:Original_aux} can be found in
\cite{flows} in full generality. Here we recall only those aspects that will be needed in the setting considered in this paper:
\begin{center}
\begin{tabular}{ c| c }
Terms in Theorem \ref{thm:Original_aux} & Specific definition in our case\\
\hline
 Metric space $X$ & $\Q_p$ \\
 Measure $\mu$ on $X$ & Haar measure $\mu$ with $\mu(\Z_p)=1$\\
 Set of valuations $S$ & $\{p,\infty\}$\\
 Parameter $m$ & $n+1$
\end{tabular}
\end{center}
Because of the ultrametric property, $\Q_p$ is a Besicovitch metric space with the Besicovitch constant $N_{\Q_p}=1$. It is also readily verified that the Haar measure on $\Q_p$ is uniformly Federer with Federer constant $D_\mu\le3p$; see \cite{flows}.

{\colorblue{}The notation $\|\cdot\|_{\mu,B}$ stands for the $\mu$-essential supremum of a function on $B$. We will also write $\|\cdot\|_{B}$ for $\|\cdot\|_{\mu,B}$ when $\mu$ is Haar measure on $\Q_p$.}

{\colorblue{}Next, the set $\Q_S$ is defined to be the direct product of the completions $\Q_v$ of $\Q$ over $v \in S$ with $\Q$ embedded into $\Q_S$ diagonally, that is, $\Q\ni r\mapsto (r,\dots,r)\in\Q_S$. 
In our case of interest where $S=\{p,\infty\}$, we have $\Q_S=\Q_p\times\R$, and so 
$$
\Q_S^{n+1}=\Q_p^{n+1}\times\R^{n+1}
$$
and
\begin{equation}\label{GL}
\GL(n+1,\Q_S)=\GL(n+1,\Q_p)\times \GL(n+1,\R).
\end{equation}}

Given $\mathbf{x} = (\mathbf{x}^{(v)})_{v\in S} \in \Q^{n+1}_S$, the quantity $c(\vv x)$, called the {\em content} of $\vv x$, is defined as
\begin{equation}\label{eqn3.4+}
c(\mathbf{x}):=\prod_{v\in S}\|\mathbf{x}^{(v)}\|_v\,,
\end{equation}
where the $v$-norm of $\vv x^{(v)}=(x^{(v)}_0,\dots,x^{(v)}_n)$ is given by
$$
\|\vv x^{(v)}\|_v=\max\{|x^{(v)}_0|_v,\dots,|x^{(v)}_n|_v\}\,.
$$
{\colorblue{}The ring $\Z_S\subset\Q$ of $S$-integers is defined as the set of $r\in\Q$ such that $|r|_q\le1$ for any prime $q\not\in S$.} For $S=\{p,\infty\}$, we have that $\Z_S=\Z[\frac1p]$. It consists of all integers and all rational numbers whose denominators are positive integer powers of $p$.
Further, $\mathfrak{B}(\Z_S,n+1)$ is the set of all non-zero primitive submodules of $\Z_S^{n+1}$. Note that if
$\Lambda$ is a discrete $\Z_S$-submodule of $\Q^{n+1}_S$, then $\Lambda$ is of the form $g\Delta$ for some $g\in\GL(n+1,\Q_S)$ and a discrete submodule $\Delta$ of $\Z_S^{n+1}$ \cite{flows}.
By \cite[Lemma~8.2]{flows}, if
$$\Lambda=\Z_S\vv a_1\oplus\cdots\oplus\Z_S\vv a_k
$$
is a $\Z_S$-submodule of $\Q_S^{n+1}$, then its (appropriately normalised) covolume can be computed as the content of the wedge product of its $\Z_S$-basis vectors:
\begin{equation}\label{eqn3.4}
\cov(\Lambda)=c(\vv a_1\wedge\dots\wedge\vv a_k)\,.
\end{equation}
Finally, for $\Lambda\subset\Q^{n+1}_S$, define
\begin{equation}
 \delta(\Lambda) := \min \left\{
 c(\mathbf{x}) : \mathbf{x} \in \Lambda \setminus \{\mathbf{0}\}
 \right\}\,.
\end{equation}
Regarding the definition of $(C,\alpha)$-good functions, used in Theorem~\ref{thm:Original_aux}, we refer to \cite{flows}. In the application of Theorem~\ref{thm:Original_aux} considered in this paper, the corresponding functions will always be polynomials. Our needs will therefore be fully covered by the following lemma.

\begin{lem}[Lemma 3.4 of \cite{flows}]\label{lem:polys_are_good}
Let $F$ be either $\R$ or a locally compact ultrametric valued field. Then, for any $d,k \in \N$, any polynomial $f\in F[X_1,\dots,X_d]$ of degree not greater than $k$ is $(C,1/dk)$-good on $F^d$ with respect to Haar measure, where $C$ is a constant depending only on $d$ and $k$.
\end{lem}

The proof of this lemma for the case $d=1$ considered in this paper can also be found in \cite[Lemma 4.1]{good_polys_Tomanov}. 
We can now specialise Theorem~\ref{thm:Original_aux} to the setting of this paper.

\begin{corr}\label{mine}
Let $\mu$ be the Haar measure on $\Q_p$ normalised so that $\mu(\Z_p)=1$, $S= \{p,\infty\}$, and $h:\tilde{B} \rightarrow \GL(n+1, \Q_S)$ be a map, where $B:=B(x_0,r)$ and $\tilde{B} = B(x_0,3^{n+1}r)$ are balls in $\Q_p$.
Suppose that for some $C, \alpha >0$ and $0<\rho <1$ one has
\begin{enumerate}
\item[{\rm(1)}] for all $\Delta\in \mathfrak{B}(\Z_S,n+1)$, the function $\operatorname{cov}(h(\cdot)\Delta)$ is $(C,\alpha)$-good on $\tilde{B}$;
 \item[{\rm(2)}] for all $\Delta\in \mathfrak{B}(\Z_S,n+1)$, $\|\operatorname{cov}(h(\cdot )\Delta)\|_{B} \ge \rho$.
 \end{enumerate}
Then, for every $0<\varepsilon \le \rho$,
\begin{equation}
 \mu \left( \{x \in B : \delta(h(x)\Z_S^{n+1})\le \varepsilon\}\right)
 \le
 C(n+1)(3p)^{2(n+1)} \left(\frac{\varepsilon}{\rho}\right)^{\alpha} \mu(B).
\end{equation}
\end{corr}

{\colorblue{}\subsection{The map $h$}\label{themaph}

In this paper, we will apply Corollary~\ref{mine} to the map $h$ defined as follows.} For every $x\in\Q_p$ define
\begin{align}
 h_1(x) &=
 \begin{pmatrix}
 g_0 & 0 & \cdots & 0\\
 0 & g_1 & \cdots & 0 \\
 \vdots & \vdots & \ddots & \vdots \\
 0 & 0 & \cdots & g_n
 \end{pmatrix}
 \begin{pmatrix}
 1 & x & \cdots & x^n\\
 0 & 1 & \cdots & nx^{n-1} \\
 \vdots & \vdots & \ddots & \vdots \\
 0 & 0 & \cdots & 1
 \end{pmatrix}, \label{h1_Matrix}
 \\[1ex]
 h_2(x) &=
 d \cdot I_{n+1}\label{h2_Matrix},
\end{align}
{\colorblue{}where $g_0,\dots,g_n$ and $d$ are integer powers of $p$, and so $g_i=|g_i|_p^{-1}$ for each $i$.} Additionally, we will require that
\begin{equation}\label{eq:gidi=1}
 d^{n+1}\prod_{i=0}^n|g_i|_p =1\,.
\end{equation}
Finally, define
\begin{equation} \label{eq:full_h_map}
 h:=(h_1,h_2):\Q_p \rightarrow \GL(n+1,\Q_S),
\end{equation}
{\colorblue{}where $S=\{p,\infty\}$ and $\GL(n+1,\Q_S)$ is given by \eqref{GL}.
Thus, $h$ assigns to each $x\in\Q_p$ the pair $(h_1(x),h_2(x))$. We note that, by \eqref{h2_Matrix}, the second component of $h$, namely $h_2$, does not depend on $x$.}

{\colorblue{}In what follows,
$\xi_0, \dots, \xi_n, Q$ satisfy \eqref{xiorder}, \eqref{xiorder+} and \eqref{eqn2.6}. Let $C_2$ be as in \eqref{c2}. We will use two cases for the choice of parameters $g_i$ and $d$.

\medskip

\noindent\underline{Case 1}. Let $j\in\{0,\dots,n\}$, 
\begin{equation}\label{delta0}
\delta_0=p^{2(n+1)r}\quad\text{ for some $r\in\Z$,}
\end{equation}
$\delta_*$ be defined from the equation
\begin{equation}\label{delta*}
\delta_0=\delta_*^{2(n+1)}C_2^{-n-1},
\end{equation}
and $\delta_i^j$ be given by \eqref{eq:delta_def}.
In view of \eqref{c2} and \eqref{delta0}, $\delta_*$ is an integer power of $p$. Further, let
\begin{align}
 |g_i|_p &= \frac{\delta_*}{\delta_i^j\xi_i} = \begin{cases}
 \dfrac{\delta_*^{1-{2(n+1)}}C_2^{n+1}}{\xi_i} \qquad &\text{if } i=j, \\
 \dfrac{\delta_*}{\xi_i} \qquad &\text{otherwise}\,,
 \end{cases} \label{defn:g_i}
 \\
 d &= \frac{\delta_*}{C_2Q}\,. \label{defn:d}
\end{align}

Using \eqref{xiorder+} and \eqref{eq:delta_def} it is easily seen that condition \eqref{eq:gidi=1} is fulfilled. Next, using \eqref{h1_Matrix}, \eqref{h2_Matrix}, \eqref{defn:g_i} and \eqref{defn:d}, we see that \eqref{eq:improved_Polys} together with the condition $H(P)\le C_2Q$ is equivalent to the system
 \begin{equation}\label{h1a less than delta}
 \left\|h_1(x)\vv a\right\|_p \le \delta_*\,,\qquad
 \left\|h_2(x)\vv a \right\|_\infty \le \delta_*\,,
 \end{equation}
where $\vv a=(a_0,\dots,a_n)^T$ is the column vector of the coefficients of $P$.
Then, using the definition of $E_{j}(B;\delta_0)$, given by \eqref{Ej}, we obtain the following proposition.

\begin{prop}\label{prop:start of QND}
Let $\xi_0, \cdots, \xi_n, Q$ satisfy \eqref{xiorder}, \eqref{xiorder+} and \eqref{eqn2.6}. Let $j\in\{0,\dots,n\}$, $C_2$ be as in \eqref{c2}, $\delta_0$ satisfy \eqref{delta0} and
$\delta_*$ be given by \eqref{delta*}. Let $h_1$ and $h_2$ be given by \eqref{h1_Matrix} and \eqref{h2_Matrix} with $g_i$ and $d$ given by \eqref{defn:g_i} and \eqref{defn:d}. Then
\begin{equation}\label{set1}
E_{j}(B;\delta_0)\;\subset\; \{x\in B:\exists\;\vv a\in\Z^{n+1}_{\neq\vv0}\;\;\text{satisfying }\eqref{h1a less than delta}\}.
\end{equation}
\end{prop}

\bigskip

\noindent\underline{Case 2}. Let $\delta_*$ be an integer power of $p$ and}
\begin{align}
 |g_i|_p &= \frac{\delta_*}{\xi_i} \label{defn:g_iQ},
 \\
 d &= \frac{1}{\delta_* Q} \label{defn:dQ}.
\end{align}
{\colorblue{}Using \eqref{xiorder+} it is easily seen that condition \eqref{eq:gidi=1} is fulfilled.

Next, using \eqref{h1_Matrix}, \eqref{h2_Matrix}, \eqref{defn:g_iQ} and \eqref{defn:dQ}, we see that \eqref{polybounds oldNEW} together with {\colorblue{}the assumption}
\begin{align}
 H(P)=\max_{0\le i \le n} |a_i| &\le \delta_*^2 Q\label{eq2.26}
\end{align}
is equivalent to the system \eqref{h1a less than delta}, where $\vv a=(a_0,\dots,a_n)^T$ is the column vector of the coefficients of $P$.
Then, setting 
\begin{equation}\label{delta*2}
\delta_*^2=\varepsilon_0
\end{equation}
and using the definition of $E(B;\varepsilon_0)$, given by \eqref{vb6.4}, we obtain the following proposition.

\begin{prop}\label{prop:start of alt QND}
Let $\xi_0, \cdots, \xi_n, Q$ satisfy \eqref{xiorder}, \eqref{xiorder+} and \eqref{eqn2.6}. Let $\delta_* >0$ be an integer power of $p$ such that $\delta_*^2=\varepsilon_0$. Let $h_1$ and $h_2$ be given by \eqref{h1_Matrix} and \eqref{h2_Matrix} with $g_i$ and $d$ given by \eqref{defn:g_iQ} and \eqref{defn:dQ}. Then
\begin{equation}\label{set2}
E(B;\varepsilon_0)\subset \{x\in B:\exists\;\vv a\in\Z^{n+1}_{\neq\vv0}\;\;\text{satisfying }\eqref{h1a less than delta}\}.
\end{equation}
\end{prop}

The aim is now to show that, in both cases, the map $h$ as defined in \eqref{h1_Matrix}--\eqref{eq:full_h_map} satisfies the hypotheses of Corollary \ref{mine}. This will enable us to use Corollary \ref{mine} to estimate the measure of the sets in \eqref{set1} and \eqref{set2} from above, and consequently derive the measure estimate \eqref{muB_G}.}

\subsection{Verifying conditions (1) and (2) in Corollary \ref{mine}}

{\colorblue{}Condition (1) in Corollary \ref{mine} will follow from Lemma \ref{lem:polys_are_good} once we establish that the coordinates of the relevant multivectors are polynomial functions of $x$. The main task of this subsection is therefore to verify condition (2) of Corollary \ref{mine}.}
We begin with auxiliary statements (Proposition~\ref{prop:ordering} and Corollaries~\ref{corr:product_of_dg_i_is_bigger_than_1} and \ref{corr:product_of_dg_i_is_bigger_than_1-w}) regarding the parameters $g_i$ and $d$ defined in \S\ref{themaph}.

\begin{prop}\label{prop:ordering}
For $0 \le i \le n$, let $g_i$ and $d$ be integer powers of $p$ satisfying \eqref{eq:gidi=1}.
Further, suppose that for some parameters $s_1,\dots,s_{n}\ge s_0:=1$, we have that
\begin{equation}\label{eq3.6}
s_ig_i \le s_{i+1}g_{i+1}\qquad\text{for $0 \le i \le n-1$.}
\end{equation}
Then for all $1 \le k \le n$
\begin{equation}\label{eq:the_max_product_gid}
 \left(\prod^{k-1}_{i=0}d|g_i|_p\right)^{-1} \le \max \left\{ \frac{1}{d|g_0|_p},|g_n|_pd\prod_{i=1}^{n-1}s_i\right\}\,.
 \end{equation}
\end{prop}
\begin{proof}
First note that $(\prod^{k-1}_{i=0}d|g_i|_p)^{-1}=\prod^{k-1}_{i=0}d^{-1}g_i$ since each $g_i$ is a power of $p$.
Using the inequalities $s_ig_i \le s_{i+1}g_{i+1}$ we get that
\begin{equation}
 \frac{g_0}{d} \le \frac{s_1g_1}{d}\le \dots \le \frac{s_ng_n}{d}.
\end{equation}
Define $j_0$ (if it exists) to be the minimum of all possible $j$ such that $s_jg_jd^{-1}\ge1$. Then it is readily seen that there are four different types of behaviour of the product $\Pi_k:=\prod^k_{i=0} s_ig_id^{-1}$ as a function of $k$, summarized in the figure below. In each case the maximal value of the product (over $0\le k\le n-1$) is achieved at either $k=0$ or $k=n-1$.
\begin{figure}[h]
\centering
 \begin{subfigure}[t]{0.23\textwidth}
 \begin{tikzpicture}[scale=0.5]
 \draw[thick,->] (0,0) -- (4.5,0)node[anchor=north west] {$k$};
\draw[thick,->] (0,0) -- (0,4.5);
\draw [thick] (4,0.2) parabola (0.2,4);
\end{tikzpicture}
 \caption{$j_0$ does not\newline exist}
 \end{subfigure}
\hfill
 \begin{subfigure}[t]{0.23\textwidth}
 \begin{tikzpicture}[scale=0.5]
 \draw[thick,->] (0,0) -- (4.5,0)node[anchor=north west] {$k$};
\draw[thick,->] (0,0) -- (0,4.5);
\draw [thick] (0.2,0.2) parabola (4,4);
\end{tikzpicture}
 \caption{$j_0=0$}
 \end{subfigure}
\hfill
 \begin{subfigure}[t]{0.23\textwidth}
 \begin{tikzpicture}[scale=0.5]
 \draw[thick,->] (0,0) -- (4.5,0)node[anchor=north west] {$k$};
 \draw[thick,->] (0,0) -- (0,4.5);
 \draw [thick] (0.2,4) .. controls (2,0) .. (4,3);
 \end{tikzpicture}
 \caption{$j_0>0$,\newline $\Pi_0 \ge \Pi_{n-1}$}
 \end{subfigure}
\hfill
 \begin{subfigure}[t]{0.23\textwidth}
 \begin{tikzpicture}[scale=0.5]
 \draw[thick,->] (0,0) -- (4.5,0)node[anchor=north west] {$k$};
\draw[thick,->] (0,0) -- (0,4.5);
\draw [thick] (0.2,3) .. controls (2,0) .. (4,4);
\end{tikzpicture}
 \caption{$j_0>0$,\newline $\Pi_0 \le \Pi_{n-1}$}
 \end{subfigure}
\end{figure}
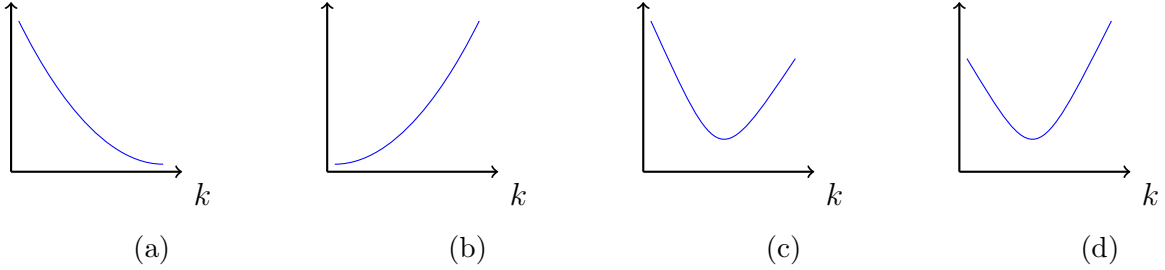

\noindent Formally, we have that
\begin{equation}\label{eq:lessthansmallerthan}
 \frac{g_0}{d} \ge \frac{g_0}{d} \cdot \frac{s_1g_1}{d} \ge \cdots
 \ge
 \prod_{i=0}^{j_0-1}\frac{s_ig_i}{d}
 \le
 \prod_{i=0}^{j_0}\frac{s_ig_i}{d}
 \le \dots \le
 \prod_{i=0}^{n-1}\frac{s_ig_i}{d}.
\end{equation}
Then the largest value of $\prod^{k-1}_{i=0} s_ig_id^{-1}$ must be $\max\{g_0d^{-1},\prod_{i=0}^{n-1}\frac{s_ig_i}{d}\}$. Since, by \eqref{eq:gidi=1}, $\prod_{i=0}^ng_id^{-1}=1$ and $g_i=|g_i|_p^{-1}$ we obtain \eqref{eq:the_max_product_gid}. If $j_0$ does not exist, then we just have the left part of \eqref{eq:lessthansmallerthan} so that the maximal value is $g_0d^{-1}$ and we again obtain \eqref{eq:the_max_product_gid}.
\end{proof}

We now specialise Proposition~\ref{prop:ordering} further for the specific values of $|g_i|_p$ and $d$ in the two cases considered in Section~\ref{themaph}. We start with Case 1.

\begin{corr} \label{corr:product_of_dg_i_is_bigger_than_1}
Let $C_2>1$, $0<\delta_*<1$, and $\xi_0,\dots,\xi_n,Q$ be integer powers of $p$ and \eqref{xiorder}, \eqref{xiorder+} and \eqref{eqn2.6} hold, where $0<v<1$.
Let $0 \le j \le n$ and $d$ and $g_i$ for $0 \le i \le n$ be given by \eqref{defn:g_i} and \eqref{defn:d}. Then for every $1\le k \le n$\begin{equation}\label{eq:bigger_than_1}
 \prod_{i=0}^{k-1}d|g_i|_p \ge Q^v\delta_*^{4n+2}C_2^{-2n-1}\,.
 \end{equation}
\end{corr}
\begin{proof}
Using \eqref{defn:g_i} and \eqref{defn:d} with $\delta_i^j$ defined by equation \eqref{eq:delta_def} it can be easily seen that
 \begin{align}
 \frac{1}{d|g_0|_p} \label{eqn3.14} &=
 \begin{cases}
 \dfrac{C_2Q\xi_0\delta_*^{2(n+1)}}{\delta_*^2C_2^{n+1}}
 & \text{if $j=0$},\\[3ex]
 \dfrac{C_2Q\xi_0}{\delta_*^2} & \text{otherwise}.
 \end{cases}
 \stackrel{\eqref{xiorder+}}{\le} \begin{cases}
 \dfrac{\delta_*^{2n}}{Q^vC_2^{n}}
 & \text{if $j=0$},\\[3ex]
 \dfrac{C_2}{Q^v\delta_*^2} & \text{otherwise}.
 \end{cases}
 \\[2ex]
 d|g_n|_p \label{eqn3.15}&=
 \begin{cases}
 \dfrac{\delta_*^2C_2^{n+1}}{C_2Q\xi_n\delta_*^{2(n+1)}}
 &\text{if $j=n$},\\[3ex]
 \dfrac{\delta_*^2}{C_2Q\xi_n} &\text{otherwise}.
 \end{cases}
 \stackrel{\xi_n=1}{=}
 \begin{cases}
 \dfrac{C_2^{n}}{Q \delta_*^{2n}}
 &\text{if $j=n$},\\[2ex]
 \dfrac{\delta_*^2}{C_2Q} &\text{otherwise}.
 \end{cases}
 \end{align}
Then, by \eqref{xiorder}, inequalities \eqref{eq3.6} are fulfilled with $(s_1,\dots,s_n)=(1,\dots,1)$ if $j=0$ and with
$$
(s_1,\dots,s_n)=(\underbrace{1,\dots,1}_{j-1},\delta_*^{-2(n+1)}C_2^{n+1},\underbrace{1,\dots,1}_{n-j})\quad\text{if $j>0$}\,.
$$
Combining \eqref{eqn3.14} and \eqref{eqn3.15} with Proposition~\ref{prop:ordering} and using the fact that $0<\delta_*<1$, $0<v<1$ and $C_2>1$ we obtain that
\begin{equation*}
 \left(\prod^{k-1}_{i=0}d|g_i|_p\right)^{-1} \le
 \max
 \left\{
 \frac{C_2}{Q^v\delta_*^2},
 \frac{C_2^{n}}{Q\delta_*^{2n}} \prod_{i=1}^{n-1}s_i
 \right\}\le \frac{C_2^{2n+1}}{Q^v\delta_*^{4n+2}}\,,
\end{equation*}
implying \eqref{eq:bigger_than_1}, as required.
\end{proof}

The following statement is an analogue of Corollary~\ref{corr:product_of_dg_i_is_bigger_than_1} for Case~2.

\begin{corr} \label{corr:product_of_dg_i_is_bigger_than_1-w}
Let $0<\delta_*<1$, and $\xi_0,\dots,\xi_n,Q$ be integer powers of $p$ and \eqref{xiorder}, \eqref{xiorder+} and \eqref{eqn2.6} hold, where $0<v<1$.
Let $d$ and $g_i$, for $0 \le i \le n$, be given by \eqref{defn:g_iQ} and \eqref{defn:dQ}. Then for every $1\le k \le n$
\begin{equation}\label{eq:bigger_than_1-2}
 \prod_{i=0}^{k-1}d|g_i|_p \ge Q^v\,.
 \end{equation}
\end{corr}

\begin{proof}
The proof of this is similar to that of Corollary~\ref{corr:product_of_dg_i_is_bigger_than_1}.
Using \eqref{defn:g_iQ} and \eqref{defn:dQ} it can be easily seen that
 \begin{align}
 \frac{1}{d|g_0|_p} \label{eqn3.14adjust} &=
 \xi_0Q
 \le Q^{-v}
 \\
 d|g_n|_p \label{eqn3.15adjust}&=
 \frac{1}{Q\xi_n}
 =
 Q^{-1}
 \end{align}
Then, by \eqref{xiorder}, inequalities \eqref{eq3.6} are fulfilled with $(s_1,\dots,s_n)=(1,\dots,1)$.
Combining \eqref{eqn3.14adjust} and \eqref{eqn3.15adjust} with Proposition~\ref{prop:ordering} we obtain that
\begin{equation*}
 \left(\prod^{k-1}_{i=0}d|g_i|_p\right)^{-1} \le
 \max
 \left\{
 \frac{1}{Q^v},
 \frac{1}{Q}
 \right\}=\frac{1}{Q^v}\,,
\end{equation*}
implying \eqref{eq:bigger_than_1-2}, as required.
\end{proof}

Now we embark upon verifying the conditions of Corollary~\ref{mine} for $h$ given by \eqref{h1_Matrix}--\eqref{h2_Matrix}.
{\colorblue{}Following \cite[\S7.3]{flows}, we use the standard basis of $\bigwedge^k(\Q_S^{n+1})$ given by
\begin{equation}\label{standard_basis}
\mathbf{e}_I := \mathbf{e}_{i_1} \wedge \cdots \wedge \mathbf{e}_{i_k} \quad \text{with } I = \{i_1 < \dots < i_k\} \subset \{0, \dots, n\},
\end{equation}
where $\mathbf{e}_0,\dots,\mathbf{e}_n$ is the standard basis of $\Q_S^{n+1}$. In particular, any $\vv w\in\bigwedge^k(\Q_S^{n+1})$ can be written as 
$$
\mathbf{w}=\sum_{I = \{i_1 < \dots < i_k\} \subset \{0, \dots, n\}} w_I \mathbf{e}_I,
$$
where $w_I=(w_I^{(v)})_{v\in S}\in\Q_S$ is the $\mathbf{e}_I$-coordinate of $\vv w$. Note that the coordinates $w_I$ have components $w^{(v)}_I$ corresponding to each $v\in S$. In the following statement (Proposition~\ref{lem:determinant is non-zero}) we provide the description of these components -- see \eqref{eq5.29} and \eqref{eqn3.17} below -- for the multivectors of the form 
\begin{equation}\label{ha}
h(x)\vv a_1\wedge\dots \wedge h(x)\vv a_k
\end{equation}
where $\vv a_1,\dots,\vv a_k\in \Z_S^{n+1}$ is a basis of some $\Z_S$-module, $S=\{p,\infty\}$ and $h=(h_1,h_2)$ is given by \eqref{h1_Matrix}--\eqref{h2_Matrix}.
In view of the definitions given in \S\ref{KT0}, in particular \eqref{eqn3.4+} and \eqref{eqn3.4}, Proposition~\ref{lem:determinant is non-zero} is a crucial step for verifying conditions (1) and (2) of Corollary~\ref{mine}.}

\begin{prop}\label{lem:determinant is non-zero}
Let $\Delta \in \mathfrak{B}(\Z_S,n+1)$ and $\vv a_1,\dots,\vv a_k$ be a basis of $\Delta$, let $h_1$ and $h_2$ be given by \eqref{h1_Matrix} and \eqref{h2_Matrix} respectively. Then
\begin{equation}\label{eq5.29}
h_2\vv a_1\wedge\dots\wedge h_2\vv a_k=d^k(\vv a_1\wedge\dots\wedge \vv a_k).
\end{equation}
{\colorblue{}Furthermore, let $l$ be the smallest integer such that $p^l(\vv a_1\wedge\dots\wedge\vv a_k)$ has integer coordinates. Then for every
$I=\{i_1<\dots<i_k\}\subset\{0,\dots,n\}$ the $\vv e_I$}-coordinate of $h_1(x)\vv a_1\wedge\dots\wedge h_1(x)\vv a_k$ in the standard basis \eqref{standard_basis} is 
\begin{equation}\label{eqn3.17}
 \left(\prod_{i\in I}g_{i}\right) p^{-l}R_I(x)
\end{equation}
{\colorblue{}for some polynomial $R_I\in\Z[X]$} of degree $\le M=\left[(\frac{n+1}{2})^2\right]$ and height
\begin{equation}\label{eqn3.21}
H(R_I)\ll\|p^l(\vv a_1\wedge\dots\wedge\vv a_k)\|_\infty.
\end{equation}
Moreover, the polynomial $R_I$ is non-zero when $I=\{0,\dots,k-1\}$.
\end{prop}

\begin{proof}
First, we note that \eqref{eq5.29} is an immediate consequence of the fact that $h_2\vv a_i=d\vv a_i$ for every $i$. Now, consider the matrix
\begin{equation}
 A=\begin{pmatrix}
 a_{0,1} & a_{0,2} & \cdots & a_{0,k}\\
 a_{1,1} & a_{1,2} & \cdots & a_{1,k}\\
 \vdots & \vdots & \ddots & \vdots \\
 a_{n,1} & a_{n,2} & \cdots & a_{n,k}\\
 \end{pmatrix}
\end{equation}
of the coordinates of $\vv a_1,\dots,\vv a_k$. Then, the coordinates of $h_1(x)\vv a_1\wedge\dots\wedge h_1(x)\vv a_k$ in the standard basis are the determinants $\det(h_{1,I}(x)A)$, where $I=\{i_1<\dots<i_k\}\subset\{0,\dots,n\}$ and $h_{1,I}(x)$ is the matrix composed of the rows $i_1,\dots,i_k$ from $h_1(x)$.

When $I=\{0, \dots, k-1\}$, it is readily seen that
\begin{align}
&\det\big(h_{1,I}(x)A\big)=\nonumber\\ &\;\;=\det\begin{pmatrix}\label{eq:RHS of h_1(x)Delta}
 g_0 P_1(x) & g_0 P_2(x) & \cdots & g_0 P_k(x)\\
 g_1 P'_1(x) & g_1 P'_2(x) & \cdots & g_1 P'_k(x)\\
 \vdots & \vdots & \ddots & \vdots \\
 \frac{g_{k-1}}{(k-1)!}P^{(k-1)}_1(x) & \frac{g_{k-1}}{(k-1)!}P^{(k-1)}_2(x) & \cdots & \frac{g_{k-1}}{(k-1)!}P^{(k-1)}_k(x)\\
 \end{pmatrix}\,,
\end{align}
where $P_j(x)=\sum_{i=0}^n a_{i,j}x^i$. It can be easily seen that the right hand-side of \eqref{eq:RHS of h_1(x)Delta} is a constant times the Wronskian of $P_1,\dots,P_k$ so we know it is non-zero. This follows from the fact that $P_1,\dots,P_k$ are linearly independent over $\R$, and this is because $\vv a_1,\dots,\vv a_k\in(\Z[\frac1p])^{n+1}$ are linearly independent vectors.

By the Laplace identity \cite[p.~105]{Schmidt-book}, we also have that
\begin{equation}\label{eqn3.23}
 \det\big(h_{1,I}(x)A\big)=(g_{i_1}\vv r_{i_1} \wedge \dots \wedge g_{i_k} \vv r_{i_k}) \cdot (\vv a_1 \wedge \dots \wedge\vv a_k)\,,
\end{equation}
where $\vv r_i$ is the $i$-th row of $h_1(x)$.
Expanding $\vv r_{i_1} \wedge \dots \wedge \vv r_{i_k}$ in the standard basis, we get a vector of $N= \binom{n+1}{k}$ polynomials, say $\hat Q_1,\dots,\hat Q_N\in\Z[X]$, of degree at most
$$
n+\dots+(n+1-k)-1-\dots-(k-1)
\le \textstyle \left[(\frac{n+1}{2})^2\right]=M\,.
$$
Then we can write $\hat Q_j(x)=\sum_{i=0}^{M} \hat{q}_{i,j}x^i$ for
$1 \le j\le N$ where $\hat{q}_{i,j}\in \Z$ depend only on $n$ and $k$. In turn, we can write $\vv a_1 \wedge \dots \wedge \vv a_k=(\hat a_1,\dots,\hat a_N)$ in the standard basis, where $\hat{a}_j \in \Z\big[\frac{1}{p}\big]$ for each $j$. By definition, $l$ is the smallest integer such that
\begin{equation}\label{eq:ahats}
 (\hat b_1,\dots,\hat b_N):=p^l(\hat{a}_1, \dots, \hat{a}_N)\in\Z^N\,.
\end{equation}
Hence, by \eqref{eqn3.23} and \eqref{eq:ahats},
\begin{equation}
\begin{aligned}
\det \big(h_{1,I}(x)A\big) &=
 \left(\prod_{i\in I} g_i\right)(\hat Q_1(x) ,\dots , \hat Q_N(x))\cdot(\hat{a}_1, \dots, \hat{a}_N) \\
&=
 \left(\prod_{i\in I}g_i\right)p^{-l}(\hat Q_1(x) ,\dots , \hat Q_N(x))\cdot(\hat{b }_1, \dots, \hat{b}_N) \\
&= \left(\prod_{i\in I}g_i\right)p^{-l}\sum_{j=1}^N \hat{b}_j \hat{Q}_j(x) \\
 &= \left(\prod_{i\in I}g_i\right) p^{-l}\sum_{j=1}^N \hat{b}_j \sum_{i=0}^{M} \hat{q}_{i,j}x^i \\
 &= \left(\prod_{i\in I}g_i\right) p^{-l}\sum_{i=0}^M c_ix^i,\quad\text{where }c_i:=\sum_{j=1}^N \hat{b}_j \hat{q}_{i,j}\,.
\end{aligned}
\end{equation}
Define the polynomial
\begin{equation}\label{eq:R(x)}
 R_I(X):=\sum_{i=0}^M c_iX^i\,.
\end{equation}
Clearly $R_I \in \Z[X]$.
Finally, it can be easily seen that
\begin{equation*}
 |c_i| \le
 \sum_{j=1}^N \left|\hat{b}_j \hat{q}_{i,j}\right| \ll_{n} \max_j |\hat{b}_j|
 = \|(\hat b_1,\dots,\hat b_N)\|_\infty=\|p^l(\vv a_1\wedge\dots\wedge\vv a_k)\|_\infty\,,
\end{equation*}
whence \eqref{eqn3.21} follows.
\end{proof}

\begin{prop}\label{lem:determinant is non-zero - together}
Let $\Delta \in \mathfrak{B}(\Z_S,n+1)$ be of rank $k$, and $h_1$ and $h_2$ be given by \eqref{h1_Matrix} and \eqref{h2_Matrix}.
Then
\begin{equation}\label{eq:Bound_on_cov(h(x)Delta)}
 \cov(h(x)\Delta) \gg
 \left(\prod_{i=0}^{k-1}d|g_i|_p\right) |\tilde R(x)|_p
 \end{equation}
for some $\tilde R\in\Z_S[x]$ such that
 \begin{equation}\label{eq3.32}
 \tilde R=\sum_{i=0}^M \tilde c_ix^i\quad\text{with}\quad
 \max_i |\tilde c_i|_p=1\,.
 \end{equation}
\end{prop}

\begin{proof}
Using the same notation as in the proof of Proposition~\ref{lem:determinant is non-zero}, let $I=\{0,\dots,k-1\}$, where $k=\operatorname{rank}\Delta$. {\colorblue{}Let $\vv a_1,\dots,\vv a_k$ be a $\Z[\frac1p]$-basis of $\Delta$. Then,} by Proposition~\ref{lem:determinant is non-zero}, \eqref{eqn3.4+} and \eqref{eqn3.4}, we have that
\begin{equation}\label{eqnn3.32}
\begin{aligned}
 \operatorname{cov}(h(x)\Delta) &\ge |\det(h_{1,I}(x)A)|_p\cdot \|d^k(\vv a_1 \wedge \dots \wedge \vv a_k)\|_\infty,
 \\
 &= \left|\left(\prod_{i=0}^{k-1}g_i\right) p^{-l}R_I(x)\right|_p \cdot
 \|d^kp^{-l}p^l(\vv a_1 \wedge \dots \wedge \vv a_k)\|_\infty,
 \\
 &= \left(\prod_{i=0}^{k-1}d|g_i|_p\right) \left|R_I(x)\right|_p \|p^l(\vv a_1 \wedge \dots \wedge \vv a_k)\|_\infty.
\end{aligned}
\end{equation}
As in the proof of Proposition~\ref{lem:determinant is non-zero}, let $c_j$ denote the coefficients of $R_I$, so that $R_I$ is given by \eqref{eq:R(x)}. Let $\tilde C=\max_j|c_j|_p$. {\colorblue{}Clearly $\tilde C$ is an integer power of $p$ and therefore $|\tilde C|_p=\tilde C^{-1}$.} Define
$$
\tilde R(X):=R_I(X)\tilde C=\sum_{i=0}^M \tilde c_iX^i\,,\quad\text{where }\tilde c_i=c_i\tilde C\,.
$$
Note that
\begin{equation}
 \max_i |\tilde c_i|_p=\max_i |c_i\tilde C|_p = \max_i|c_i|_p\tilde C^{-1}=1\,.
\end{equation}
Since $|c_i|_p|c_i|\ge1$, we have that $|c_i|_p\|\vv c\|_\infty=|c_i|_pH(R_I)\ge1$. Therefore, $\tilde C H(R_I)\ge1$ and, by \eqref{eqn3.21}, we get that
\begin{equation}\label{eqnn3.34}
\tilde C\cdot\|p^l(\vv a_1 \wedge \dots \wedge \vv a_k)\|_\infty\gg 1\,.
\end{equation}
Observe that
\begin{equation}\label{eqnn3.35}
 |R(x)|_p = \left|\tilde R(x)\tilde C^{-1} \right|_p
 = \left| \tilde{R}(x)\right|_p \cdot \tilde C\,.
\end{equation}
Then using \eqref{eqnn3.32},\eqref{eqnn3.34} and \eqref{eqnn3.35} we obtain that
\begin{align*}
 \operatorname{cov}(h(x)\Delta)
 &\ge
 \left(\prod_{i=0}^{k-1}d|g_i|_p\right) |\tilde{R}(x)|_p \cdot \tilde C \cdot
 \|p^l(\vv a_1\wedge\dots\wedge\vv a_k)\|_\infty\\[1ex]
 & \gg \left(\prod_{i=0}^{k-1}d|g_i|_p\right) |\tilde R(x)|_p
\end{align*}
as required.
\end{proof}

\begin{prop}\label{lem:specifc h map}
Let $\delta_*$, $Q$, $\xi_0,\dots,\xi_n$, $g_0,\dots,g_n$, $d$ be as in Corollary~\ref{corr:product_of_dg_i_is_bigger_than_1} or Corollary~\ref{corr:product_of_dg_i_is_bigger_than_1-w}. Let $\rho=1$ and $\alpha = M^{-1}$, where $M=[\left(\frac{n+1}{2}\right)^2]$.
Then, for any non-empty ball $B\subset\Z_p$ and all sufficiently large $Q$,
the map $h$ given by \eqref{h1_Matrix}--\eqref{h2_Matrix} satisfies the conditions stated in Corollary \ref{mine}, in which $C>0$ depends on $n$ only.
\end{prop}

\begin{proof}
The validity of condition~(1) in Corollary~\ref{mine} follows from Lemma~\ref{lem:polys_are_good}. Indeed, by Proposition~\ref{lem:determinant is non-zero} and the definition of $\operatorname{cov}(h(\cdot)\Delta)$, the function $\operatorname{cov}(h(\cdot)\Delta)$ is the maximum of $p$-adic absolute values of polynomials in one variable of degree at most $M$, and therefore, by Lemma~\ref{lem:polys_are_good} and \cite[Lemma 3.1]{flows}, it is $(C,\alpha)$ good for $\alpha=M^{-1}$ and some $C>0$ depending only on $M$. Thus, ultimately $C$ depends on $n$ only.

Now we verify condition (2) in Corollary~\ref{mine}. Fix any non-empty ball $B\subset\Z_p$. If $k=n+1$ then, since $\prod_{i=0}^n(d|g_i|_p)=1$, using the explicit form of $h_1$ and $h_2$ given by \eqref{h1_Matrix} and \eqref{h2_Matrix} one readily verifies that $\cov(h(x)\Delta)=1\ge\rho$. Indeed, since $\Delta$ has rank $n+1$ and is primitive, the standard basis $\vv e_0,\dots,\vv e_n$ is a basis of $\Delta$. Then $\|h_1(x)\vv e_1\wedge\dots\wedge h_1(x)\vv e_{n+1}\|_p=\prod_{i=0}^n|g_i|_p$ and $\|h_2(x)\vv e_1\wedge\dots\wedge h_2(x)\vv e_{n+1}\|_\infty=d^{n+1}$. Then
\begin{align*}
\cov(h(x)\Delta)&=\|h_1(x)\vv e_1\wedge\dots\wedge h_1(x)\vv e_{n+1}\|_p\times\\ &\times\|h_2(x)\vv e_1\wedge\dots\wedge h_2(x)\vv e_{n+1}\|_\infty
=\prod_{i=0}^n(d|g_i|_p)=1\,,
\end{align*}
as claimed above.

Naturally, for the rest of the proof we will assume that $1\le k\le n$. By \eqref{eq:Bound_on_cov(h(x)Delta)} we have that
\begin{equation}\label{eqn3.37}
 \|\cov(h(x)\Delta)\|_B \gg
 \left(\prod_{i=0}^{k-1}d|g_i|_p\right) \sup_{x\in B}|\tilde R_{\tilde{\vv c}}(x)|_p\,,
 \end{equation}
where $\tilde{\vv c}=(\tilde c_0,\dots,\tilde c_M)\in\Z[\frac1p]^{M+1}$ and $\tilde R_{\tilde{\vv c}}$ satisfies \eqref{eq3.32}.
Define
\begin{equation}\label{tilderho}
 \tilde{\rho}:=\inf_{\|\tilde{\vv c}\|_p=1}\sup_{x \in B} |\tilde R_{\tilde{\vv c}}(x)|_p\,.
\end{equation}
Clearly $\tilde{\rho}$ is a constant depending on $k$, $n$, $p$ and $B$ only. Since $B$ is non-empty, we have that for every choice of $\tilde{\vv c}\in\Q_p^{M+1}$ with $\|\tilde{\vv c}\|_p=1$ we have that
\begin{equation}\label{eqn3.39}
\sup_{x\in B}|\tilde R_{\tilde{\vv c}}(x)|_p
\end{equation}
is strictly positive.
Also, since for every fixed $x\in\Q_p$, $\tilde R_{\tilde{\vv c}}(x)$ is a linear function of $\tilde{\vv c}$, we have that \eqref{eqn3.39} depends on $\tilde{\vv c}$ continuously. Since the set of $\tilde{\vv c}\in\Q_p^{M+1}$ subject to $\|\tilde{\vv c}\|_p=1$ is compact, we conclude that $\tilde \rho$, given by \eqref{tilderho}, is strictly positive.

Now, combining \eqref{eqn3.37} and \eqref{tilderho}, and using Corollary~\ref{corr:product_of_dg_i_is_bigger_than_1} and Corollary~\ref{corr:product_of_dg_i_is_bigger_than_1-w} together with the facts that $\delta_*\le1$ and $C_2\ge1$, we obtain that
$$
\|\cov(h(x)\Delta)\|_B \gg Q^v\delta_*^{4n+2}C_2^{-2n-1}\tilde\rho\,,
$$
where the implied constant depends on $n$ only. Therefore, since $\delta_*$, $C_2$ and $\tilde\rho$ do not depend on $Q$, we have that
$$
\|\cov(h(x)\Delta)\|_B\ge\rho=1
$$
provided that $Q$ is sufficiently large.
\end{proof}

Combining Proposition~\ref{lem:specifc h map} with Corollary~\ref{mine} with any $0<\rho<1$ we obtain the following

\begin{corr}\label{aux}
Let $n\ge2$ be an integer, $p$ be a prime number, $\mu$ be Haar measure on $\Q_p$. Let $\delta_*$, $Q$, $v$, $\xi_0,\dots,\xi_n$, $g_0,\dots,g_n$, $d$ be as in Corollary~\ref{corr:product_of_dg_i_is_bigger_than_1} or Corollary~\ref{corr:product_of_dg_i_is_bigger_than_1-w}. Let $\alpha = [\left(\frac{n+1}{2}\right)^2]^{-1}$ and $h$ be given by \eqref{h1_Matrix}--\eqref{h2_Matrix}.
Then there exists a constant $K>0$ depending on $n$ and $p$ only satisfying the following statement. For any non-empty ball $B\subset\Z_p$ there exists $Q_0=Q_0(B,n,p,v,C_2)$ such that for all $Q\ge Q_0$ and $\varepsilon>0$ one has that
 \begin{equation} \label{eq:auxstatement}
 \mu \left( \{x \in B : \delta(h(x)\Z_S^{n+1})\le\varepsilon\}\right)
 \le
 K\varepsilon^\alpha \mu(B).
 \end{equation}
\end{corr}

We remark that the constant $K$ appearing in \eqref{eq:auxstatement} is given by $$K=C(n+1)(3p)^{2(n+1)},$$ where $C$ arises from condition (2) of Corollary~\ref{mine} and, as established in Proposition~\ref{lem:specifc h map}, depends only on $n$ and $p$.

{\colorblue{}\subsection{Completion of the proof of Lemma~\ref{lem:aux_lemma}}

Let $h$ be given by \eqref{h1_Matrix}--\eqref{h2_Matrix}, where $d,g_0,\dots,g_n$ are given by either \eqref{defn:g_i}--\eqref{defn:d} or \eqref{defn:g_iQ}--\eqref{defn:dQ}. In view of the definition of the content $c(\cdot)$, given by \eqref{eqn3.4+}, covolume $\delta(\cdot)$, given by \eqref{eqn3.4}, and the fact that $\Z\subset\Z_S=\Z[\frac1p]$, we have that \eqref{h1a less than delta} implies that
$$
\delta(h(x)\Z_S^{n+1})\le\delta_*^2,
$$
where $\delta_*$ is given by either \eqref{delta*} or \eqref{delta*2} depending on whether we are in Case~1 or Case~2. In either case, by Proposition~\ref{prop:start of QND} and Proposition~\ref{prop:start of alt QND}, we have that the sets
$E_j(B;\delta_0)$ and $E(B;\varepsilon_0)$ are contained in 
$$
\{x \in B : \delta(h(x)\Z_S^{n+1})\le \delta_*^2\}
$$
for one of the choices of $d,g_0,\dots,g_n$ above. Consequently, by Corollary~\ref{aux}, we obtain that, in Case 1, for each $0\le j\le n$
\begin{equation}
 \mu(E_j(B;\delta_0))
 \le K\delta_*^{2\alpha} \mu(B)\;\stackrel{\eqref{delta*}}{=}\;
 K\delta_0^{\alpha/(n+1)}C_2^{\alpha} \mu(B)
\end{equation}
and, in Case 2, 
\begin{align}
 \mu(E(B; \varepsilon_0))
 \le K\delta_*^{2\alpha} \mu(B)\;\stackrel{\eqref{delta*2}}{=}\;
 K\varepsilon_0^{\alpha} \mu(B)\,,
\end{align}
for sufficiently large $Q$.

Choosing $\varepsilon_0$ and $\delta_0$ small enough so that
\begin{equation}\label{eps0}
\max\left\{K\varepsilon_0^{\alpha},\;K\delta_0^{\alpha/(n+1)}C_2^{\alpha}\right\} \le \frac{1-\kappa}{n+2}
\end{equation}
ensures that
$$
 \mu(E(B; \varepsilon_0))
 \le
 \frac{1-\kappa}{n+2}\, \mu(B)\quad\text{and}\quad \mu(E_j(B;\delta_0))\le \frac{1-\kappa}{n+2}\, \mu(B).
$$
It is also straightforward to see that $\varepsilon_0$ and $\delta_0$ can be chosen to depend on $n$, $p$, and $\kappa$ only.

Furthermore, by \eqref{G_B}, we get that
\begin{equation}\label{equation:meausre of G_B}
\begin{aligned}
 \mu(G_B)
 &\ge
 \mu(B) - \sum_{j=0}^n \mu(E_j(B;\delta_0))
 -\mu(E(B;\varepsilon_0)) \\
 &\ge
 \mu(B) - (n+2)\frac{1-\kappa}{n+2}\mu(B)
 =
 \kappa\mu(B).
\end{aligned}
\end{equation}
This verifies \eqref{muB_G} and thus completes the proof of Lemma~\ref{lem:aux_lemma}.}

\section{Finding close roots}
In this section we will establish how close to $x$ the roots of a polynomial satisfying system \eqref{eq:aux_lemma_statement} are. The parameters $\xi_i$ will be suitably chosen. We will use Hensel's Lemma, which can be found, for example, in \cite{Hensel's_Lemma}, to identify a suitable root $\alpha\in\Q_p$ of $P$ close to $x$.

\begin{lem}[Hensel's Lemma]
Let $f \in \Z_p[x]$, $x \in \Z_p$ and $|f(x)|_p < |f'(x)|^2_p$.
Then there exists a unique $\alpha \in \Z_p$ such that $f (\alpha) = 0$, $|f'(\alpha)|_p = |f'(x)|_p$, and
$$|x - \alpha|_p = |f(x)|_p\cdot|f'(x)|_p^{-1}<|f'(x)|_p.$$
\end{lem}

Now we specialise Hensel's Lemma to the setup of Lemma~\ref{lem:aux_lemma}.

\begin{corr}\label{lem:bounds_on_first_two_xis}
Let $n\ge2$, $0<\delta_0<1$, $Q>1$ and $\xi_0,\dots,\xi_n>0$. Suppose that
\begin{equation}\label{eq:Condition on b_1 and b_2}
 \xi_0 < (\delta_0\xi_1)^2\,.
\end{equation}
Let $x\in\Z_p$. Then, for any $P \in \mathcal{P}_n(Q)$ satisfying \eqref{eq:aux_lemma_statement}
there exists a unique root $\alpha\in\Z_p$ of $P$ such that
\begin{equation}\label{eqn5.2}
 \begin{aligned}
 |x-\alpha|_p \le
 \delta_0^{-1}\xi_0\xi_1^{-1}.
 \end{aligned}
\end{equation}
\end{corr}

\begin{proof}
With $f=P$, \eqref{eq:aux_lemma_statement} and
\eqref{eq:Condition on b_1 and b_2} verify the condition $|f(x)|_p < |f'(x)|^2_p$ in Hensel's Lemma, and therefore \eqref{eqn5.2} follows immediately.
\end{proof}

\begin{lem} \label{lem:with_ordered_roots}
Let $x\in \Z_p$ and $P\in \Z_p[x]$ be a polynomial of degree $n\ge2$, with the leading coefficient $a_n$ and roots $\alpha_1, \dots, \alpha_n\in\overline{\Q_p}$ ordered so that
\begin{equation}\label{eq:ordered_roots}
 |x-\alpha_1|_p
 \le
 |x-\alpha_2|_p
 \le
 \cdots
 \le
 |x-\alpha_n|_p.
 \end{equation}
Then, for any $0 \le j<n$, the following bound holds
\begin{equation}\label{eq:derivartive_bounded_by_roots}
 \left|\tfrac{1}{j!}P^{(j)}(x)\right|_p \le |a_n|_p |x-\alpha_{j+1}|_p \cdots |x-\alpha_n|_p.
 \end{equation}
Furthermore, {\colorblue{}for $1 \le j<n$,} if $|x-\alpha_{j}|_p < |x-\alpha_{j+1}|_p$ then we have equality in \eqref{eq:derivartive_bounded_by_roots}.
\end{lem}

\begin{proof}
Write the polynomial $P$ as the product
$P(X)=a_n(X-\alpha_1)\cdots(X-\alpha_n)$.
Then on differentiating this expression we obtain that
\begin{equation} \label{eq:Sum_of_Roots}
\tfrac1{j!}P^{(j)}(x) = a_n\sum_{1\le i_1 < \cdots <i_{n-j} \le n}
 (x-\alpha_{i_1}) \cdots (x-\alpha_{i_{n-j}}).
 \end{equation}
 Define $T_{j+1}=(x-\alpha_{j+1}) \cdots (x-\alpha_{n})$. By \eqref{eq:ordered_roots}, $T_{j+1}$ has the largest $p$-adic value in the sum of \eqref{eq:Sum_of_Roots}.
 We will also define $\widehat{T}_{j+1}$ to be the term with the second largest $p$-adic value in the sum.
 The $p$-adic value of each term in the sum in \eqref{eq:Sum_of_Roots} is less than or equal to $|T_{j+1}|_p$.
 Hence by the ultrametric property it must be that
\begin{equation}\label{eq:derivartive_bounded_by_roots(Short)}
 \left|\tfrac1{j!}P^{(j)}(x)\right|_p \le |a_n|_p|T_{j+1}|_p\,,
 \end{equation}
which is exactly \eqref{eq:derivartive_bounded_by_roots}. Next, we can rewrite equation \eqref{eq:Sum_of_Roots} as
\begin{equation}\label{eq:Sum_of_Roots_rewritten}
\tfrac1{j!}P^{(j)}(x) = a_n\left(\sum_{1\le i_1 < \cdots <i_{n-j} \le n}\;
 \prod_{\ell=1}^{n-j}(x-\alpha_{i_\ell}) - T_{j+1} + T_{j+1}\right).
 \end{equation}
By the ultrametric property again, we must have that
 \begin{equation} \label{eq:bound_on_largest_p_adic_term}
 \left|\sum_{1\le i_1 < \cdots <i_{n-j} \le n}\;
 \prod_{\ell=1}^{n-j}(x-\alpha_{i_\ell}) - T_{j+1}\right|_p \le \left|\widehat{T}_{j+1}\right|_p,
 \end{equation}
 as by taking away the largest term we must be left with the second largest term.
 Observe that $|x-\alpha_{j}|_p < |x-\alpha_{j+1}|_p$ implies that $|\widehat T_{j+1}|_p<|T_{j+1}|_p$, and therefore by, \eqref{eq:Sum_of_Roots_rewritten}, \eqref{eq:bound_on_largest_p_adic_term} and the ultrametric property, we obtain that $|\tfrac1{j!}P^{(j)}(x)|_p=|a_n|_p|T_{j+1}|_p$. This means exactly the equality in \eqref{eq:derivartive_bounded_by_roots}.
\end{proof}

\begin{lem}\label{lem:bounds_on_all_roots}
Let $x\in \Z_p$ and $Q>1$.
Let $P \in \mathcal{P}_n(Q)$ be such that inequalities \eqref{eq:aux_lemma_statement} hold with $\xi_i= Q^{-\theta_i}$ for some $\theta_i$, where $0 \le i \le n$. Let $\alpha_1, \dots, \alpha_n\in\overline{\Q_p}$ be the roots of $P$ ordered as in Lemma~\ref{lem:with_ordered_roots}.
Define
\begin{equation}\label{d_j_values_in_terms_of_b}
d_j=\theta_{j-1}-\theta_j
\end{equation}
for $1 \le j \le n$ and suppose that
\begin{equation}\label{eq:d_i's}
 d_1 \ge d_2 \ge \cdots \ge d_n \ge 0\,.
 \end{equation}
Then the roots of $P$ satisfy the inequalities
\begin{equation}\label{eqn5.12}
 |x-\alpha_j|_p \le \delta_0^{-1} Q^{-d_j} \qquad (1 \le j \le n).
\end{equation}
\end{lem}

\begin{proof}
We will prove \eqref{eqn5.12} by induction on $j$. First consider $j=1$. Then, using \eqref{eq:derivartive_bounded_by_roots}, we obtain that
\begin{equation}
 \left|P'(x)\right|_p \le |a_n|_p |x-\alpha_{2}|_p \cdots |x-\alpha_n|_p = \frac{|P(x)|_p}{|x-\alpha_{1}|_p}.
\end{equation}
By rearranging and using the bounds from equation \eqref{eq:aux_lemma_statement} we obtain that
\begin{equation}
 |x-\alpha_{1}|_p \le \frac{|P(x)|_p}{|P'(x)|_p}
 \le \frac{Q^{-\theta_0}}{\delta_0Q^{-\theta_1}}
 = \delta_0^{-1}Q^{-d_1}
\end{equation}
as required in \eqref{eqn5.12} for $j=1$.

Now suppose that $1\le j<n$ and \eqref{eqn5.12} holds for this $j$. We shall prove \eqref{eqn5.12} for $j+1$.
Define $T_{j+1}=(x-\alpha_{j+1})\cdots (x-\alpha_{n})$ and $T_{j+2}=(x-\alpha_{j+2})\cdots (x-\alpha_{n})$, as in Lemma \ref{lem:with_ordered_roots}, where $T_{j+2}=1$ if $j=n-1$.
By Lemma \ref{lem:with_ordered_roots}, we get that
\begin{equation}\label{eq:derivartive_bounded_by_roots(Adjusted)}
\begin{aligned}
 \left|\tfrac{1}{(j+1)!} P^{(j+1)}(x)\right|_p \cdot |x-\alpha_{j+1}|_p
 &\le |a_n|_p |T_{j+2}|_p |x-\alpha_{j+1}|_p
 = |a_n|_p |T_{j+1}|_p\,,
\end{aligned}
\end{equation}
and so
\begin{equation}\label{eqn5.16}
\begin{aligned}
 |x-\alpha_{j+1}|_p
 \le \frac{|a_n|_p |T_{j+1}|_p}{\left|\frac{1}{(j+1)!} P^{(j+1)}(x)\right|_p}\,.
\end{aligned}
\end{equation}
If additionally, we assume that $|x-\alpha_{j}|_p < |x-\alpha_{j+1}|_p$ then, by Lemma \ref{lem:with_ordered_roots}, we obtain that $|a_n|_p|T_{j+1}|_p=\left|\tfrac{1}{j!}P^{(j)}(x)\right|_p$ and we obtain from \eqref{eqn5.16} and \eqref{eq:aux_lemma_statement}
that
\begin{equation}
 \begin{aligned}
 |x-\alpha_{j+1}|_p
 &\le \frac{\left|\frac{1}{j!} P^{(j)}(x)\right|_p}{\left|\frac{1}{(j+1)!} P^{(j+1)}(x)\right|_p}
 \le \frac{Q^{-\theta_j}}{\delta_0 Q^{-\theta_{j+1}}}
 = \delta_0^{-1} Q^{-d_{j+1}}.
 \end{aligned}
\end{equation}
If $|x-\alpha_{j}|_p < |x-\alpha_{j+1}|_p$ does not hold, then, by \eqref{eq:ordered_roots}, we have that $|x-\alpha_{j}|_p = |x-\alpha_{j+1}|_p$. Using \eqref{eq:d_i's} and the induction assumption, we then get that
\begin{equation}
 \begin{aligned}
 |x-\alpha_{j+1}|_p
 =
 |x-\alpha_{j}|_p \le \delta_0^{-1}Q^{-d_j}
 \le \delta_0^{-1}Q^{-d_{j+1}}\,,
 \end{aligned}
\end{equation}
thereby proving the required statement for $j+1$ and finishing the proof.
\end{proof}

\section{Root separation: proof of Theorem \ref{Thm:Main2}}

Let $n \ge 2$, $p$ be a prime, $v=1$, $0<\kappa<1$ and $\delta_0$, $C_1$ and $C_2$ be the constants arising from Lemma~\ref{lem:aux_lemma}. Take any ball $B\subset\Z_p$ and let $Q>Q_0$, where $Q_0$ is again as in Lemma~\ref{lem:aux_lemma}.

Let $\theta$ satisfy equation \eqref{eq:Beta_1_bound}. Define $\xi_2=\dots=\xi_n = 1$,
$$
\xi_0=\left\{\begin{array}{cl}
 \delta_0Q^{-n-1+\theta} &\text{if }\theta>1\,, \\[1ex]
 Q^{-n-1+\theta} &\text{if }\theta\le1\,,
\end{array}\right.\quad\text{and}\quad
\xi_1=\left\{\begin{array}{cl}
 \delta_0^{-1}Q^{-\theta} &\text{if }\theta>1\,, \\[1ex]
 Q^{-\theta} &\text{if }\theta\le1\,.
\end{array}\right.
$$
Define $\theta_i$ by the equation $\xi_i=Q^{-\theta_i}$ for $0 \le i \le n$. Then, it is readily verified that
\begin{align*}
 2\le\frac{2}{3}(n+1) < \theta_0 \le n+1 \qquad\text{and}\qquad
 0 \le \theta_1 < \frac{n+1}{3}
\end{align*}
and that \eqref{eq:Condition on b_1 and b_2} holds for all sufficiently large $Q$.

Then, clearly \eqref{xiorder} and \eqref{xiorder+}${}_{v=1}$ hold and Lemma~\ref{lem:aux_lemma} is applicable, and we have a measurable set $G_B \subset B$ satisfying \eqref{muB_G}. Take any $x\in G_B$ and fix, by Lemma~\ref{lem:aux_lemma}, any primitive irreducible polynomials $P \in \Z[X]$ of degree $n$ and height $C_1Q\le H(P)\le C_2Q$ satisfying \eqref{eq:aux_lemma_statement}.

Let $\alpha_1, \dots \alpha_n\in\overline{\Q_p}$ be the roots of $P$ ordered as in equation \eqref{eq:ordered_roots}. It is readily seen that \eqref{eq:d_i's} holds.
Then by Lemma \ref{lem:bounds_on_all_roots} we have that
\begin{align}
 |x-\alpha_1|_p &\le \delta_0^{-1} Q^{-\theta_0 +\theta_1}\le\delta_0^{-1}Q^{-(n+1-2\theta)}\,,\label{eqn6.3} \\ \nonumber
 |x-\alpha_2|_p &\le \delta_0^{-1} Q^{-\theta_1}\le \delta_0^{-2}Q^{-\theta}\,.
\end{align}
By Corollary~\ref{lem:bounds_on_first_two_xis}, $\alpha_1$ must be the same as $\alpha$ arising from Corollary~\ref{lem:bounds_on_first_two_xis} and therefore $\alpha_1\in\Z_p$. By the ultrametric property $\alpha_1\in B$ provided that $Q$ is sufficiently large.
By \eqref{eq:Beta_1_bound} and the ultrametric property again
\begin{align}
 |\alpha_1 - \alpha_2|_p \le
 \max \{ |x- \alpha_1|_p, |x- \alpha_2|_p \}
 \le \delta_0^{-2} Q^{-\theta}.\label{eqn6.4}
\end{align}
This completes the proof of Theorem \ref{Thm:Main2}, with $C_0=\delta_0^{-2}$. Indeed, \eqref{eq:measure_of_roots_intersecting_B_is_greater_than_3/4_of_B}
follows from \eqref{eqn6.3} and \eqref{muB_G}, while \eqref{eqn6.4} together with the aforementioned properties of $P$ ensures that $\alpha=\alpha_1$ belongs to $\set$.

\section{Counting discriminants: proof of Theorem~\ref{Thm:Main}}

The proof follows the ideas of \cite{BBG16}. Let $n \ge 2$, $p$ be a prime, $v=1/n$, $\kappa=1/2$ and $\delta_0$, $C_1$ and $C_2$ be the constants arising from Lemma~\ref{lem:aux_lemma}. Take $B=\Z_p$ and let $Q>Q_0$, where $Q_0$ is again as in Lemma~\ref{lem:aux_lemma}.

Let $0\le \nu\le n-1$. Let $\theta_n=0$, $d_1,\dots,d_n$ satisfy \eqref{eq:d_i's} and let $\theta_{n-1},\dots,\theta_0$ be defined by \eqref{d_j_values_in_terms_of_b}. Clearly, we have that
\begin{equation}\label{eq:betas_structured}
 \theta_0 \ge \cdots \ge\theta_n = 0\,.
\end{equation}
We also set $\xi_i=Q^{-\theta_i}$ and require that $\theta_0+\dots+\theta_n=n+1$. By \eqref{eq:betas_structured}, we have that $\theta_0\ge1+1/n$. Hence \eqref{xiorder} and \eqref{xiorder+} with $v=1/n$ hold and Lemma~\ref{lem:aux_lemma} is applicable. Therefore, there is a measurable set $G_B \subset B$ satisfying \eqref{muB_G}, where $B=\Z_p$. Take any $x\in G_B$ and fix, by Lemma~\ref{lem:aux_lemma}, any primitive irreducible polynomial $P \in \Z[X]$ of degree $n$ and height $C_1Q\le H(P)\le C_2Q$ satisfying \eqref{eq:aux_lemma_statement}.

Let $\alpha_1, \dots \alpha_n\in\overline{\Q_p}$ be the roots of $P$ ordered such as in equation \eqref{eq:ordered_roots}.
Then by Lemma~\ref{lem:bounds_on_all_roots} and the ultrametric property we have that
\begin{equation}
 |\alpha_i - \alpha_j|_p \le
 \delta_0^{-1}Q^{-d_j}
\end{equation}
for any $1\le i<j\le n$.
It follows that
\begin{equation} \begin{aligned}\label{eq:det_1st}
 0 < |D(P)|_p \le |a_n|_p^{2n-2} \prod_{1 \le i < j \le n} Q^{-2d_j} \ll Q^{-2 \sum_{j=2}^n(j-1)d_j}.
 \end{aligned}
\end{equation}
Setting
\begin{equation}\label{eq:nu_sum}
 \nu = \sum_{j=1}^n (j-1)d_j
\end{equation}
gives that
$0 < |D(P)|_p \ll Q^{-2\nu}$.

Rearranging \eqref{d_j_values_in_terms_of_b} we get
$\theta_{j-1}= d_j + \theta_j$, and then we obtain that
$\theta_{j-1} = d_j + \cdots + d_n + \theta_n=d_j + \cdots + d_n $ since $\theta_n=0$. Hence,
\begin{equation}\label{eq:sum_of_jd_j's}
 \sum_{j=1}^njd_j = \sum_{j=1}^{n} (d_{j} + \cdots + d_n)= \sum_{j=1}^{n} \theta_{j-1}= \sum_{j=0}^{n-1
} \theta_{j}= n+1,
\end{equation}
where we have used the fact that $\theta_n=0$.
Now it is possible to compute $\nu$ by expanding the right hand-side of equation \eqref{eq:nu_sum}:
\begin{equation}\label{eqn7.6}
 \nu = n+1 - \sum_{j=1}^n d_j\,.
\end{equation}
By Lemmas~\ref{lem:aux_lemma} and \ref{lem:bounds_on_all_roots}, for every $x \in G_B$ there exists an irreducible polynomial $P\in\Z[X]$ of degree $n$ with one of its roots $\alpha=\alpha(P)$ satisfying
 \begin{equation}
 |x-\alpha(P)|_p \le \delta_0^{-1}Q^{-d_1}\,.
 \end{equation}
Hence,
 \begin{equation}\label{eq:G_J is a subset}
 G_B \subset \bigcup_{P \in \mathcal{D}_{n,p}(C_2Q,\nu)} \bigcup_{j=1}^n \left\{ x\in\Z_p:
 |x-\alpha_j(P)|_p \le \delta_0^{-1} Q^{-d_1}
 \right\}\,,
 \end{equation}
where $\alpha_1(P),\dots,\alpha_n(P)\in\overline{\Q_p}$ are the roots of $P$. Therefore, since $B=\Z_p$, we have that
 \begin{equation}
 \frac{1}{2}=\frac{1}{2} \mu(B) \le \#\mathcal{D}_{n,p}(C_2Q,\nu) \cdot n\delta_0^{-1} Q^{-d_1}
 \end{equation}
 and so by rearranging we get
\begin{equation}\label{eqn7.10}
 \#\mathcal{D}_{n,p}(C_2Q,\nu) \ge \frac{\delta_0}{2n} Q^{d_1}.
 \end{equation}
It can be further seen that the best possible lower bound is obtained by maximising the value of $d_1$, or by \eqref{eqn7.6}, minimizing $d_2,\dots,d_n$. By \eqref{eq:d_i's}, this can be done by letting $d_2=d_3 = \cdots = d_n$, and, by solving \eqref{eq:sum_of_jd_j's} and \eqref{eqn7.6}, we obtain that
 \begin{equation}
 d_1=n+1-\frac{n+2}{n}\nu \qquad \text{ and } \qquad d_2= \frac{2\nu}{n(n-1)}.
 \end{equation}
It is readily seen that $d_1\ge d_2$ for $0\le \nu \le n-1$. Substituting $d_1$ into \eqref{eqn7.10} and rescaling the bound for the height by letting $\tilde{Q}=C_2Q$ we complete the proof.

\medskip

\noindent\textit{Acknowledgements.} VB was supported by the EPSRC grant EP/Y016769/1. The authors are grateful to the anonymous reviewer for their helpful comments, which enabled us to improve the presentation and accuracy of our arguments.

\bigskip

{\small
\noindent \begin{minipage}{0.5\textwidth}Victor Beresnevich:\\ 
Department of Mathematics\\ 
University of York\\ 
Heslington, York, YO10 5DD\\
{\tt victor.beresnevich@york.ac.uk}
\end{minipage}
\qquad \begin{minipage}{0.5\textwidth}Bethany Dixon:\\ 
Department of Mathematics\\
University of York\\
Heslington, York, YO10 5DD
\end{minipage}
}

\end{document}